\documentclass[12pt]{amsart}
\title[A Tannakian approach to dimensional reduction]{A Tannakian
  approach to \\ dimensional reduction of principal bundles}

\author[L. \'Alvarez-C\'onsul]{Luis \'Alvarez-C\'onsul}
\address{Instituto de Ciencias Matem\'aticas CSIC-UAM-UC3M-UCM,
Calle Nicol\'as Cabrera, 13-15, Campus Cantoblanco UAM, 28049
Madrid, Spain}
\email{l.alvarez-consul@icmat.es}
\author[I. Biswas]{Indranil Biswas}
\address{School of Mathematics, Tata Institute of
Fundamental Research, Homi Bhabha Road, Mumbai 400005, India}
\email{indranil@math.tifr.res.in}
\author[O. Garc\'{i}a-Prada]{Oscar Garc\'{i}a-Prada}
\address{Instituto de Ciencias Matem\'aticas CSIC-UAM-UC3M-UCM,
Calle Nicol\'as Cabrera, 13-15, Campus Cantoblanco UAM, 28049
Madrid, Spain}
\email{oscar.garcia-prada@icmat.es}

\dedicatory{To Ugo Bruzzo on his 60th birthday}



\usepackage{amssymb}
\usepackage{hyperref,url}
\usepackage{bm} 
\usepackage{bbm} 
\usepackage{mathrsfs} 

\newtheorem{theorem}{Theorem}[section]
\newtheorem{proposition}[theorem]{Proposition}
\newtheorem{lemma}[theorem]{Lemma}
\newtheorem{corollary}[theorem]{Corollary}

\theoremstyle{definition}
\newtheorem{definition}[theorem]{Definition}
\newtheorem{remark}[theorem]{Remark}
\newtheorem{example}[theorem]{Example}

\numberwithin{equation}{section}

\newcommand{\secref}[1]{Section~\ref{#1}} 

\newcommand{\defeq}{\mathrel{\mathop:}=} 

\newcommand{\Hom}{\operatorname{Hom}}
\newcommand{\End}{\operatorname{End}}
\newcommand{\Aut}{\operatorname{Aut}}
\newcommand{\ev}{\operatorname{ev}}
\newcommand{\VectC}{{\operatorname{Vect}_\CC}}
\newcommand{\LieH}{\operatorname{Lie}(H)}
\newcommand{\ad}{\operatorname{ad}}

\newcommand{\Vol}{\operatorname{Vol}}
\newcommand{\rk}{\operatorname{rk}}

\newcommand{\GL}{\operatorname{GL}}
\newcommand{\SL}{\operatorname{SL}}
\newcommand{\Id}{\operatorname{Id}}

\newcommand{\CC}{{\mathbb{C}}}
\newcommand{\FF}{{\mathcal{E}}}
\newcommand{\UU}{{\mathbb{U}}}
\newcommand{\VV}{{\mathbb{V}}}

\newcommand{\PP}{{\mathbb{P}}}
\newcommand{\RR}{{\mathbb{R}}}
\newcommand{\cC}{{\mathcal{C}}}
\newcommand{\cF}{{\mathcal{F}}}
\newcommand{\cG}{{\mathcal{G}}}
\newcommand{\cK}{{\mathcal{K}}}
\newcommand{\cO}{{\mathcal{O}}}
\newcommand{\cS}{{\mathcal{S}}}

\newcommand{\RRR}{{\mathscr{R}}}

\newcommand{\glg}{\mathfrak{g}}
\newcommand{\glp}{\mathfrak{p}}
\newcommand{\glu}{\mathfrak{u}}
\newcommand{\gll}{\mathfrak{l}}
\newcommand{\glr}{\mathfrak{r}}
\newcommand{\glt}{\mathfrak{t}}

\renewcommand{\mod}[1]{\operatorname{mod}(#1)}
\newcommand{\Rep}[1]{\operatorname{Rep}(#1)}

\renewcommand{\bf}{{\bm{f}}}
\newcommand{\bh}{{\bm{h}}}
\newcommand{\bE}{{\bm{E}}}

\newcommand{\bU}{{\bm{U}}}
\newcommand{\bV}{{\bm{V}}}

\newcommand{\bvarphi}{{\bm{\varphi}}}
\newcommand{\bphi}{{\bm{\phi}}}
\newcommand{\bpsi}{{\bm{\psi}}}
\newcommand{\bepsilon}{{\bm{\varepsilon}}}
\renewcommand{\epsilon}{\varepsilon}

\newcommand{\dual}[1]{#1^\vee}
\newcommand{\duali}[2]{#1^{\vee#2}}

\begin{document}

\thanks{Partially supported by the the European Commission Marie Curie
  IRSES MODULI Programme IRSES-GA-2013-612534. The first and the third
  authors were partially supported by the Spanish MINECO under the
  ICMAT Severo Ochoa grant No. SEV-2011-0087, 
  and grant No.~MTM2013-43963-P}

\begin{abstract}
Let $P$ be a parabolic subgroup of a connected simply connected complex semisimple Lie group
$G$. Given a compact K\"ahler manifold $X$, the
dimensional reduction of $G$-equivariant holomorphic vector bundles
over $X\times G/P$ was carried out by the first and third
authors~\cite{AG2}. This raises the question of dimensional reduction of
holomorphic principal bundles over $X\times G/P$. The method used for
equivariant vector bundles does not generalize to principal bundles.
In this paper, we adapt to equivariant principal bundles the Tannakian
approach of Nori, to describe the dimensional reduction of
$G$-equivariant principal bundles over $X\times G/P$, and to establish
a Hitchin--Kobayashi type correspondence. In order to be able to apply
the Tannakian theory, we need to assume that $X$ is a complex
projective manifold.
\end{abstract}

\maketitle

\section{Introduction}

Dimensional reduction is a very powerful construction in the context
of gauge theory, both in physics and geometry. Many important
gauge-theoretic equations appear as symmetric solutions of fundamental
equations for connections, like the instanton equations on Riemannian
$4$-manifolds.  Examples of these are the Bogomol'nyi equations for
magnetic monopoles in 3 dimensions and the vortex equation in 2
dimensions, as well as many other important integrable systems and
soliton equations.  In the context of model building, some physicists
have applied for a long time the method of `coset-space dimensional
reduction' in the construction of gauge unified theories (see, e.g.,
\cite{forgacs-manton,bais-et-al,kapetanakis-zoupanos,szabo-valdivia}).

In \cite{Ga0,Ga,BG,AG1} the dimensional reduction techniques were
brought into the context of holomorphic vector bundles over K\"ahler
manifolds, to study the dimensional reduction of stable
$\SL(2,\CC)$-equivariant bundles over $X\times \PP^1$ and the
corresponding Hermitian--Yang--Mills equations, where $X$ is a compact
K\"ahler manifold and $\PP^1$ is the Riemann sphere.  In \cite{AG2}
this construction was generalised to $G$-equivariant vector bundles on
$X\times G/P$, where $G$ is a
connected simply connected complex semisimple Lie group and $P\subset
G$ is a parabolic subgroup. These gave rise to quiver bundles with
relations, that is representations of a quiver with relations in the
category of vector bundles, where the quiver with relations $(Q,\cK)$
is determined by $P$. In particular when $X$ is a point, one has a
description of homogeneous vector bundles over $G/P$ in terms of
representations of the quiver with relations in the category of vector
spaces.

In this paper we undertake the task of generalizing \cite{AG2} to
holomorphic principal $H$-bundles over $X\times G/P$, where $H$ is a
complex reductive group. 
Since a $G$-equivariant holomorphic principal $H$-bundle over $X\times
G/P$ is equivalent to a $P$-equivariant holomorphic principal
$H$-bundle over $X$ (for the trivial $P$-action on $X$), the problem
reduces then to studying the dimensional reduction of such bundles
over $X$. To do this we take the Tannakian point of view of Nori
\cite{No}, regarding a principal $H$-bundle as a functor from the
category of representations of $H$ to the category of vector
bundles. To apply this to our situation, we first show
in~\secref{sec2} that the category of representations of $(Q,\cK)$ in
complex vector spaces has a structure of neutral Tannakian
category. The construction of an identity object, dual objects and
tensor products is extended to the category of $(Q,\cK)$-bundles
in~\secref{sec:equivpfb-Qbun}, where we also show that $P$-equivariant
holomorphic principal $H$-bundles on $X$ are in bijection with
$H$-torsors in the category of $(Q,\cK)$-bundles over $X$, that is
strict exact faithful tensor functors from the category of $H$-modules
to the category $(Q,\cK)$-bundles. For our Tannakian approach we need
to assume that $X$ is projective. In~\secref{sec:dim-red} we define
stability of an $H$-torsor in the category of $(Q,\cK)$-bundles, and
show that the semistability (resp. polystability) of such a torsor is
equivalent to the semistability (resp. polystability) of the
corresponding $G$-equivariant holomorphic $H$-bundle on $X\times G/P$.
We complete our work by showing that the polystability of an
$H$-torsor is equivalent to the existence of solutions of the quiver
vortex equations on the $(Q,\cK)$-bundle corresponding to the adjoint
representation of $H$.

To the knowledge of the authors, this is the first paper where a
Tannakian approach to dimensional reduction is adopted. Another more
direct approach to describe  $P$-equivariant holomorphic bundles on $X$ 
in terms of pairs consisting of a holomorphic principal bundle over $X$ 
with a reduced structure
group and certain `Higgs fields' is indeed possible. We plan to come back
to this and its relation to our Tannakian approach in a future paper.  

\subsection*{Acknowledgments}

We wish to thank ICMAT, Madrid, and TIFR, Mumbai, for hospitality and
support.

\section{Homogeneous vector bundles and quiver modules}\label{sec2}

This section is devoted to a Tannakian description of holomorphic
equivariant vector bundles over a flag manifold $G/P$, in terms of
representations of a quiver with relations. 

Throughout this paper, $G$ is a connected simply connected semisimple
complex affine algebraic group, and $P\subsetneq G$ is a parabolic
subgroup with unipotent radical $R_u(P)\subset P$. The Lie algebras of
$R_u(P)\subset P\subset G$ are denoted $\glu\subset\glp\subset\glg$,
respectively. We also fix a Levi subgroup $L$ of $P$, that is, a
maximal connected reductive subgroup of $P$.

\subsection{Induction and reduction for homogeneous vector bundles}
\label{sub:ind-red-homogeneous-vect}

Let $G/P$ be the quotient of $G$ under the right $P$-action $G\times
P\to G, (g,p)\mapsto gp$ (its points are the orbits $gP$, for $g\in
G$). It is a flag variety, i.e., a complex projective manifold, with a
transitive holomorphic left $G$-action $G\times G/P\to G/P,
(g,g'P)\mapsto gg'P$.
A \emph{homogeneous holomorphic vector bundle} on $G/P$ is a
holomorphic vector bundle $E$ on $G/P$, with an equivariant
holomorphic $G$-action, that is, a holomorphic $G$-action via
holomorphic vector-bundle automorphisms such that $g\cdot v\in
E_{gg'P}$ for $g,g'\in G$, $v\in E_{g'P}$.

In Proposition~\ref{prop:lem1}, we will provide an algebraic
description of homogeneous holomorphic vector bundle in terms of
quivers with relations. Our first step for this description is
provided by a standard process of induction and reduction.

\begin{lemma}\label{lem:ind-red-homgeneous-vect}
There is an equivalence of categories between the homogeneous
holomorphic vector bundles on $G/P$, with $G$-equivariant holomorphic
vector-bundle maps, and the (holomorphic finite-dimensional)
representations of $P$, with $P$-equivariant linear maps.
\end{lemma}

\begin{proof}
One first observes that the projection $G\to G/P$ is a holomorphic
principal $P$-bundle. Then the equivalence maps a homogeneous
holomorphic vector bundle $E$ on $G/P$ into its fibre $E_o$ over the
base point $o=P\in G/P$ with isotropy group $P$, and a representations
$V$ of $P$ into its associated holomorphic vector bundle
$E=G\times^PV$, with $G$-action induced by left multiplication of $G$
on itself. See, e.g.,~\cite[\S 1.1.1]{AG2} for details.
\end{proof}

\subsection{Quivers, relations, and their modules}
\label{sub:rep-quiver-relation}

To fix notation, here we recall a few relevant definitions from the
theory of quiver modules (see, e.g.,~\cite{ASS} for an introduction to
this topic). A \emph{quiver} $Q$ is a pair of sets $Q_0$ and $Q_1$,
together with two maps $t,h\colon Q_1\to Q_1$.  In this paper, we do
not impose the usual restriction that the sets $Q_0$ and $Q_1$ are
finite. The elements of $Q_0$ and $Q_1$ are called the \emph{vertices}
and the \emph{arrows} of the quiver, respectively. An arrow $a\in Q_1$
is represented pictorially as $a\colon v\to w$, where $v=ta$ and
$w=ha$ are called the \emph{tail} and the \emph{head} of the quiver,
respectively. Using the convention that arrows compose as maps, from
right to left, a \emph{non-trivial path} in $Q$ of length $\ell\geq 1$
is a sequence $p=a_\ell\cdots a_1$ of arrows $a_i\in Q_1$ that
compose, that is, $ha_{i-1}=ta_{i}$ for all $1\leq i\leq\ell$. It is
represented as
\begin{equation}\label{eq:path}
p\colon
\stackrel{hp}{\bullet}
\stackrel{a_\ell}{\longleftarrow}
\cdots
\stackrel{a_2}{\longleftarrow}
\bullet
\stackrel{a_1}{\longleftarrow}
\stackrel{tp}{\bullet},
\end{equation}
where $tp=ta_1$ and $hp=ha_\ell$ are called the \emph{tail} and the
\emph{head} of the path $p$. The \emph{trivial path} $e_v$ at a vertex
$v\in Q_0$ consists of the zero-length path at vertex $v$ with no
arrows, so it has tail and head $te_v=he_v=v$. A (complex)
\emph{relation} of $Q$ with tail $v$ and head $w$ is a finite formal
sum
\begin{equation}\label{eq:relation}
r=c_1p_1+\cdots+c_kp_k,
\end{equation}
of paths $p_1,\ldots,p_k$ such that $tp_i=v$ and $hp_i=w$ for all
$1\leq i\leq k$, with coefficients $c_i\in\CC$. A \emph{quiver with
  relations} is a pair $(Q,\cK)$ consisting of a quiver $Q$ and a set
$\cK$ of relations of $Q$. A (complex) linear representation of $Q$,
or a \emph{$Q$-module} for short, is a pair $(\bV,\bvarphi)$
consisting of a set $\bV$ of finite-dimensional vector spaces $V_v$,
indexed by the vertices $v\in Q_0$, and a set $\bvarphi$ of linear
maps $\varphi_a\colon V_{ta}\to V_{ha}$, indexed by the arrows $a\in
Q_1$. Since $Q_0$ may be infinite, we impose the condition that
$V_v=0$ for all but finitely many $v\in Q_0$. A path $p$ of $Q$
determines a linear map, for each $Q$-module $(\bV,\bvarphi)$,
\[
\varphi(p)\colon V_{tp}\longrightarrow V_{hp},
\]
defined as $\varphi(p)=\varphi_{a_\ell}\circ\cdots\circ\varphi_{a_1}$ for a
non-trivial path~\eqref{eq:path}, and as the identity
$\varphi(p)=\Id_{V_v}\colon V_v\to V_v$ for a trivial path $p=e_v$ at a
vertex $v\in Q_0$. Such a $Q$-module $(\bV,\bvarphi)$ \emph{satisfies a
  relation~\eqref{eq:relation}} if the linear map $\varphi(r)\defeq
c_1\varphi(p_1) +\cdots+ c_k\varphi(p_k)$ is zero.

Quiver modules form a category $\mod{Q}$, where a morphism
$\bf\colon(\bV,\bvarphi)\to(\bV',\bvarphi')$ between two $Q$-modules
is a set of linear maps $f_v\colon V_v\to V_v'$, indexed by the
vertices $v\in Q_0$, such that $\varphi_a'\circ
f_{ta}=f_{ha}\circ\varphi_a$ for all $a\in Q_1$. Morphism composition
in this category is defined by vertexwise composition of linear
maps. In this paper, we will be interested in certain full subcategory
\[
\mod{Q,\cK}\subset\mod{Q}
\]
of \emph{$(Q,\cK)$-modules}, that is, $Q$-modules that satisfy a fixed
set of relations $\cK$ of $Q$.

\subsection{The quiver with relations associated to a parabolic subgroup}
\label{sub:quiver-relation}

Here we review a construction of a quiver with relations $(Q,\cK)$
associated to $P$, such that homogeneous holomorphic vector bundles
over $G/P$ are equivalent to $(Q,\cK)$-modules (see
Remark~\ref{rem:equiv-homog-vb-QK-mod}). We follow the construction by
the first and third authors~\cite{AG2}, as it is well suited for
dimensional reduction, but we should emphasize that other
constructions exist (see, e.g.,~\cite{BK,Hille,OR}).
Recall that $G$ is a complex semisimple simply connected affine
algebraic group, $P\subsetneq G$ is a parabolic subgroup, with
unipotent radical $R_u(P)$, $L\subset P$ is a Levi subgroup, and
$\glu\subset\glp\subset\glg$ are the Lie algebras of $R_u(P)\subset
P\subset G$, respectively
(see~\secref{sub:ind-red-homogeneous-vect}). Given two (complex
finite-dimensional) vector spaces $U$ and $V$, $\Hom(U,V)$ is the
vector space of linear maps $U\to V$, and $V^*=\Hom(V,\CC)$ is the
dual vector space of $V$. If $U$ and $V$ are (holomorphic)
$L$-representations, then $\Hom(U,V)$ and $V^*=\Hom(V,\CC)$ become
$L$-representations (for the trivial $L$-action on $\CC$); in this
case, $V^L\subset V$ is the subspace of $L$-invariant vectors, and
$\Hom_L(U,V)=\Hom(U,V)^L\subset\Hom(U,V)$ is the space of
$L$-equivariant linear maps.

Let $\Lambda_P^+$ be the set of all isomorphism classes of irreducible
complex algebraic representations of $L$. Note that if $T\subset G$ is
a Cartan subgroup such that $T\subset L$, with Lie algebra
$\glt\subset\glg$, then the lattice $\Lambda\subset\glt^*$ of integral
weights of $T$ parametrizes isomorphism classes of irreducible complex
algebraic representations of $T$, and $\Lambda_P^+$ can be identified
with a fundamental chamber in $\Lambda$ for the reductive Lie group
$L$. From this view point, $\Lambda_P^+$ is identified with the set of
dominant weights of $L$.

We fix an irreducible representation $M_\lambda$ in the isomorphism
class $\lambda$, for each $\lambda\in Q_0$. For each
$\lambda,\mu,\nu\in\Lambda_P^+$, we define vector spaces
\begin{subequations}\label{eq:A-B-spaces}
\begin{align}&
A_{\mu\lambda}\defeq\Hom_L(\glu\otimes M_\lambda,M_\mu),
\label{eq:A-B-spaces.a}
\\&
B_{\mu\lambda}\defeq\Hom_L(\wedge^2\glu\otimes M_\lambda,M_\mu),
\label{eq:A-B-spaces.b}
\end{align}
\end{subequations}
where $\glu$ is regarded as an $L$-module, and linear maps
\begin{subequations}\label{eq:psi-maps-1}
\begin{gather}
\psi_{\mu\lambda}\colon A_{\mu\lambda}\longrightarrow B_{\mu\lambda},
\\
\psi_{\mu\nu\lambda}\colon A_{\mu\nu}\otimes A_{\nu\lambda}\longrightarrow A_{\mu\lambda}, 
\end{gather}
\end{subequations}
obtained by restriction (see~\cite[Lemma~1.2]{AG2}) of the linear maps
\begin{gather*}
\psi_{\mu\lambda}\colon
\Hom_L(\glu\otimes M_\lambda,M_\mu)
\longrightarrow\Hom_L(\wedge^2\glu\otimes M_\lambda,M_\mu)
\\
\psi_{\mu\nu\lambda}\colon
\Hom(\glu\otimes M_\nu,M_\mu)\otimes\Hom(\glu\otimes M_\lambda,M_\nu)
\longrightarrow \Hom(\wedge^2\glu\otimes M_\lambda,M_\mu),
\end{gather*}
(denoted with the same symbols), given by
\begin{gather*}
\psi_{\mu\lambda}(a)(e,e')\defeq-a([e,e']),\quad
\psi_{\mu\nu\lambda}(a''\otimes a')=a''\wedge a',
\end{gather*}
for all $a\in\Hom(\glu\otimes M_\lambda,M_\mu), a'\in\Hom(\glu\otimes
M_\lambda,M_\nu), a''\in\Hom(\glu\otimes M_\nu,M_\mu)$,
$e,e'\in\glu$. Note that in the definition of $\psi_ {\mu\lambda}$ and
$\psi_ {\mu\nu\lambda}$, we are using the identifications
\begin{gather*}
\Hom(\glu\otimes U,V)=\Hom(\glu,\Hom(U,V))=\glu^*\otimes\Hom(U,V),\\
\Hom(\wedge^2\glu\otimes U,V)=\Hom(\wedge^2\glu,\Hom(U,V))=\wedge^2\glu^*\otimes\Hom(U,V)
\end{gather*}
for vector spaces $U$ and $V$, and in the definition of
$\psi_{\mu\nu\lambda}$, the exterior product is
\[
a''\wedge a'\defeq(s''\wedge s')\otimes(f''\circ f')
\]
for $a'=s'\otimes f', a''=s''\otimes f''$, with $s',s''\in\glu^*$ and
$f'\in\Hom(M_\lambda,M_\nu), f''\in\Hom(M_\nu,M_\mu)$.

Let $\{a_{\mu\lambda}^{(i)}\mid i=1,\ldots,n_{\mu\lambda}\}$ and
$\{b_{\mu\lambda}^{(p)}\mid p=1,\ldots,m_{\mu\lambda}\}$ be bases of
$A_{\mu\lambda}$ and $B_{\mu\lambda}$, respectively, where
$n_{\mu\lambda}=\dim A_{\mu\lambda}$ and $m_{\mu\lambda}=\dim
B_{\mu\lambda}$, for all $\lambda,\mu\in\Lambda_P^+$. For all
$\lambda,\mu,\nu\in\Lambda_P^+$, $1\leq i\leq n_{\nu\lambda}, 1\leq
j\leq n_{\mu\nu}, 1\leq k\leq n_{\mu\lambda}$, we expand the vectors
\[
\psi_{\mu\lambda}(a_{\mu\lambda}^{(i)}),\psi_{\mu\nu\lambda}(a_{\mu\nu}^{(j)}\wedge
a_{\nu\lambda}^{(i)})\in B_{\mu\lambda}
\]
in the given basis of $B_{\mu\lambda}$ as follows, with coefficients
$c^{(k,p)}_{\mu\lambda}, c^{(j,i,p)}_{\mu\nu\lambda}\in\CC$:
\begin{equation}\label{eq:coefficientes-relations}
\psi_{\mu\lambda }(a^{(k)}_{\mu\lambda})
=\sum_{p=1}^{m_{\mu\lambda}} c^{(k,p)}_{\mu\lambda}\, b^{(p)}_{\mu\lambda},
\quad
\psi_{\mu\nu\lambda}(a^{(j)}_{\mu\nu}\otimes a^{(i)}_{\nu\lambda})
=\sum_{p=1}^{m_{\mu\lambda}} c^{(j,i,p)}_{\mu\nu\lambda}\, b^{(p)}_{\mu\lambda}.
\end{equation}

\begin{definition}\label{def:quiver-relations}
The \emph{quiver with relations $(Q,\cK)$ associated to 
$P$} is defined as follows.
\begin{itemize}
\item[(a)]
The quiver $Q$ has vertex set $Q_0=\Lambda_P^+$, i.e., the set of
isomorphism classes of irreducible representations of $L$, and arrow
set
\[
Q_1=\{a_{\mu\lambda}^{(i)}\mid \lambda,\mu\in Q_0, 1\leq i\leq n_{\mu\lambda}\}.
\]
The tail and head maps $t,h\colon Q_1\to Q_0$ are defined by
\[
t(a_{\mu\lambda}^{(i)})=\lambda,\quad h(a_{\mu\lambda}^{(i)})=\mu.
\]
\item[(b)] 
The set of relations is $\cK=\{r_{\mu\lambda}^{(p)}\mid\lambda,\mu\in
Q_0, 1\leq p\leq m_{\mu\lambda}\}$, with
\[
r_{\mu\lambda}^{(p)}\defeq
\sum_{\nu\in Q_0}\sum_{i=1}^{n_{\nu\lambda}}\sum_{j=1}^{n_{\mu\nu}}c^{(j,i,p)}_{\mu\nu\lambda}  a^{(j)}_{\mu\nu} a^{(i)}_{\nu\lambda} +\sum_{k=1}^{n_{\mu\lambda}} c^{(k,p)}_{\mu\lambda} a^{(k)}_{\mu\lambda},
\]
where $a^{(j)}_{\mu\nu} a^{(i)}_{\nu\lambda}$ denotes the path
$\mu\leftarrow\nu\leftarrow\lambda$ defined by the arrows
$a^{(j)}_{\mu\nu}$ and $a^{(i)}_{\nu\lambda}$.
\end{itemize}
\end{definition}

To simplify some notation, given a $Q$-module $(\bV,\bvarphi)$, the
linear map $\varphi_a\colon V_\lambda\to V_\mu$ corresponding to the
arrow $a=a_{\mu\lambda}^{(i)}\in A_{\mu\lambda}$ will be denoted
\begin{equation}\label{eq:phi_a}
\varphi_{\mu\lambda}^{(i)}\defeq\varphi_a\colon V_\lambda\longrightarrow V_\mu.  
\end{equation}

With Definition~\ref{def:quiver-relations}, we have the following
result, where $\Rep{P}$ is the category of (complex
finite-dimensional) representations of $P$, or $P$-modules for short,
in which the morphisms are the $P$-equivariant linear maps.

\begin{theorem}[{\cite[Theorem~1.4]{AG2}}]
\label{thm:equiv-P-mod-QK-mod}
Let $(Q,\cK)$ be the quiver with relations associated to the group
$P$. Then there is an equivalence of categories
\begin{equation}\label{eq:equiv-P-mod-QK-mod}
\Rep{P}\cong\mod{Q,\cK}.
\end{equation}
\end{theorem}

\begin{remark}\label{rem:equiv-homog-vb-QK-mod}
Since $P$-modules are equivalent to homogeneous holomorphic vector
bundles over $G/P$ (see Lemma~\ref{lem:ind-red-homgeneous-vect}),
Theorem~\ref{thm:equiv-P-mod-QK-mod} implies that the latter are also
equivalent to $(Q,\cK)$-modules (\cite[Corollary~1.13]{AG2}).
\end{remark}

For the purposes of this paper, it will be useful to sketch some
ingredients of the proof of Theorem~\ref{thm:equiv-P-mod-QK-mod}. We
begin observing that there is an equivalence
\[
\Rep{L}\cong\mod{Q_0}
\]
between $L$-modules and (finite-dimensional) $Q_0$-graded vector
spaces, i.e. sets $\bV=\{V_\lambda\}_{\lambda\in Q_0}$ of vector
spaces $V_\lambda$, indexed by the vertices $\lambda\in Q_0$, such
that $V_\lambda=0$ for all but finitely many $\lambda\in Q_0$. This
equivalence assigns to a $Q_0$-graded vector space $\bV$, the
$L$-module
\begin{equation}\label{eq:isotopical-dec-V}
  \VV=\bigoplus_{\lambda\in Q_0}V_\lambda\otimes M_\lambda,
\end{equation}
and to an $L$-module $\VV$, the $Q_0$-graded vector space $\bV$
consisting of the multiplicity spaces
$V_\lambda\defeq\Hom_L(M_\lambda,\VV)$, so
that~\eqref{eq:isotopical-dec-V} becomes the isotopical decomposition
of $\VV$. The fact that this is an equivalence follows because $L$ is
reductive. Moreover, given a $Q_0$-graded vector space $\bV$, the
choice of basis $\{a^{(i)}_{\mu\lambda}\}_{i=1}^{n_{\mu\lambda}}$ of
$A_{\mu\lambda}$ induces isomorphisms
\begin{equation}\label{eq:arrow-maps-equiv-maps}
\CC^{n_{\mu\lambda}}\otimes\Hom(V_\lambda,V_\mu)
\stackrel{\cong}{\longrightarrow}
A_{\mu\lambda}\otimes \Hom(V_\lambda,V_\mu)
\stackrel{\cong}{\longrightarrow}
\Hom_L(\glu\otimes V_\lambda\otimes M_\lambda,V_\mu\otimes M_\mu),
\end{equation}
for all $\lambda,\mu\in Q_0$, where the left-hand space parametrizes
sets
$\bvarphi_{\mu\lambda}=\{\varphi_{\mu\lambda}^{(i)}\}_{i=1}^{n_{\mu\lambda}}$
of linear maps $\varphi_{\mu\lambda}^{(i)}\in\Hom(V_\lambda,V_\mu)$,
for $1\leq i\leq n_{\mu\lambda}$, the left-hand isomorphism maps
$\bvarphi_{\mu\lambda}$ into
\[
\varphi_{\mu\lambda}\defeq\sum_{i=1}^{n_{\mu\lambda}}a^{(i)}_{\mu\lambda}\otimes
\varphi_{\mu\lambda}^{(i)}\in A_{\mu\lambda}\otimes \Hom(V_\lambda,V_\mu),
\]
and the right-hand side isomorphism follows directly from the
definition of $A_{\mu\lambda}$. Let $\VV$ be the $L$-module associated
to $\bV$ via~\eqref{eq:isotopical-dec-V}. Then there is an isomorphism
\begin{equation}\label{eq:rep-space-equiv-rep-space}
\RRR_Q(\bV)\stackrel{\cong}{\longrightarrow}\RRR_{L,\glu}(\VV),
\end{equation}
obtained taking the direct sum of the
isomorphisms~\eqref{eq:arrow-maps-equiv-maps} for all $\lambda,\mu\in
Q_0$, where the representation space of $Q$ on a $Q_0$-graded module
$\bV$ is
\begin{equation}\label{eq:rep-space}
\RRR_Q(\bV)\defeq
\bigoplus_{\lambda,\mu\in Q_0}\CC^{n_{\mu\lambda}}\otimes\Hom(V_\lambda,V_\mu)
\cong\bigoplus_{\lambda,\mu\in Q_0}A_{\mu\lambda}\otimes\Hom(V_\lambda,V_\mu),
\end{equation}
and the space of $L$-equivariant maps from $\glu$ into $\End\VV$, for
any $L$-module $\VV$, is
\begin{equation}\label{eq:equiv-rep-space}
\RRR_{L,\glu}(\VV)\defeq\Hom_L(\glu\otimes\VV,\VV)
\cong\Hom_L(\glu,\End\VV).
\end{equation}
Note that the data of a pair $(\bV,\bvarphi)$, where $\bV$ is a
$Q_0$-graded vector space and $\bvarphi\in\RRR_Q(\bV)$, is equivalent to
a $Q$-module. In fact, the
isomorphisms~\eqref{eq:rep-space-equiv-rep-space} determine an
equivalence
\begin{equation}\label{eq:intermediate-equiv}
\Rep{L,\glu}\cong\mod{Q}
\end{equation}
where the left-hand side is the category in which an object is a pair
$(\VV,\tau)$ consisting of an $L$-module $\VV$ and an $L$-equivariant
linear map $\tau\colon\glu\to\End\VV$, and a morphism
$(\VV,\tau)\to(\VV',\tau')$ is an $L$-equivariant linear map
$f\colon\VV\to\VV'$ such that $f\circ\tau(e)=\tau'(e)\circ f$ for all
$e\in\glu$. The right-hand side of~\eqref{eq:intermediate-equiv} is
the category of $Q$-modules (with no relations). To
prove~\eqref{eq:equiv-P-mod-QK-mod}, we observe now that an object of
$\Rep{P}$, i.e. a group morphism $\rho\colon P\to\GL(\VV)$, is
equivalent to an object $(\VV,\tau)$ of $\Rep{L,\glu}$ satisfying
commutation relations, where $\VV$ has $L$-action $\rho|_L\colon
L\to\GL(\VV)$, $\tau\colon\glu\to\End\VV$ is given by
\begin{equation}\label{eq:def-tau}
\tau=d\rho|_{\glu}
\end{equation}
($d\rho\colon\glp\to\End\VV$ being the differential), and the
commutation relations are
\begin{equation}\label{eq:CR2}
\tau([e,e'])=[\tau(e),\tau(e')],
\end{equation}
for all $e,e'\in\glu$. Then~\eqref{eq:equiv-P-mod-QK-mod} follows
because one can show that imposing~\eqref{eq:CR2} on an object
$(\VV,\tau)$ of $\Rep{L,\glu}$ corresponds precisely to imposing the
set of relations $\cK$ on the $Q$-module $(\bV,\bvarphi)$ corresponding
to $(\VV,\tau)$ via~\eqref{eq:intermediate-equiv} (see the proof
of~\cite[Theorem~1.4]{AG2}).

\subsection{A Tannakian structure on modules over
the quiver with relations associated to a parabolic subgroup}
\label{sub:Tannakian-QK-mod}

The category $\mod{Q,\cK}$ of modules over an arbitrary quiver with
relations $(Q,\cK)$ is always abelian, but it may have no further
sensible algebraic structure in general, such as a tensor product,
unless we impose some conditions on $(Q,\cK)$. In this section, we
construct a structure of neutral Tannakian category on $\mod{Q,\cK}$,
when $(Q,\cK)$ is the quiver with relations associated to $P$ (see,
e.g.,~\cite{CR} for research on related algebraic structures, namely
Hopf quivers).

We start with a short review of the notion of a neutral Tannakian
category (see, e.g.,~\cite{De,DM,No,Sa,Ta} for details). A \emph{tensor
  category} is a category $\cC$ with a bifunctor
$\otimes\colon\cC\times\cC\to\cC$, called the tensor product, two
natural isomorphisms $A\otimes(B\otimes C)\to(A\otimes B)\otimes C$
and $A\otimes B\to B\otimes A$, for all objects $A,B,C$ of $\cC$,
called an associativity constraint and a commutativity constraint,
satisfying certain axioms (including the pentagon and the hexagon
axioms), and with an object $\mathbbm{1}$, called the identity object,
and an isomorphism $\mathbbm{1}\to\mathbbm{1}\otimes\mathbbm{1}$, such
that $\mathbbm{1}\otimes-\colon \cC\to\cC$ is an equivalence of
categories. 
A \emph{tensor functor} $(\cF,c)\colon\cC\to\cC'$ between tensor
categories consists of a functor $\cF\colon\cC\to\cC'$ and a natural
isomorphism $c_{A,B}\colon \cF(A)\otimes \cF(B)\to \cF(A\otimes B)$
satisfying canonical axioms involving the tensor products. A
\emph{morphism} $\lambda\colon(\cF,c)\to (\cF',c')$ between two tensor
functors $(\cF,c),(\cF',c')\colon\cC\to\cC'$ is a natural
transformation $\lambda\colon \cF\to \cF'$ satisfying canonical axioms
involving tensor products. The space of these morphisms is denoted
$\Hom^\otimes(\cF,\cF')$. A tensor category is \emph{rigid} if all its
objects admit duals, i.e. there is a contravariant functor
$(-)^\vee\colon\cC^{\operatorname{op}}\to\cC$, and a natural
isomorphism $\ev_A\colon A^\vee\otimes A\to\mathbbm{1}$ for all
objects $A$ of $\cC$, inducing natural isomorphisms
$\Hom(B,A^\vee)\to\Hom(B\otimes A,\mathbbm{1})$ for all $A,B$ (then
the category has inner Hom objects $\underline{\Hom}(A,B)=B\otimes
A^\vee$). It is known that for all tensor functors
$(\cF,c),(\cF',c')\colon\cC\to\cC'$, if $\cC$ and $\cC'$ are rigid,
then every morphism $\lambda\colon(\cF,c)\to (\cF',c')$ is an
isomorphism (see, e.g.,~\cite[Proposition~1.13]{DM}). In this case, the
group of tensor automorphisms of any tensor functor
$(\cF,c)\colon\cC\to\cC'$ is
$\Aut^\otimes(\cF)=\Hom^\otimes(\cF,\cF)$, with group multiplication
given by vertical composition of natural transformations. In
Definition~\ref{def:Tannakian-category}, $\VectC$ is the tensor
category of (finite-dimensional) complex vector spaces.

\begin{definition}[{see, e.g.,~\cite[Definition~2.19]{DM}}]\label{def:Tannakian-category}
A \emph{neutral Tannakian category} (over $\CC$) is pair $(\cC,\cF)$
consisting of a rigid abelian tensor category such that
$\CC=\End(\mathbbm{1})$, and an exact faithful linear tensor functor
$\cF\colon\cC\to\VectC$, called the \emph{fibre functor}.
\end{definition}

\begin{example}\label{ex:Tannakian-Rep}
For any complex affine group scheme $\cG$, the category $\Rep{\cG}$ of
(complex finite-dimensional) representations of $\cG$ is a neutral
Tannakian category, 
with the fibre functor $\cF_\cG\colon\Rep{\cG}\to\VectC$ mapping a
$\cG$-module into its underlying vector space.
\end{example}

By the following fundamental duality theorem of
Tannaka--Grothendieck--Saavedra, roughly speaking, a neutral Tannakian
category is a category with enough added structure that it is
equivalent to the category of representations of an affine group
scheme, with the neutral Tannakian structure of
Example~\ref{ex:Tannakian-Rep}.

\begin{theorem}[see, e.g.,~{\cite[Theorem~2.11]{DM}}]\label{thm:Tannaka-Grothendieck-Saavedra}
Let $(\cC,\cF)$ be a neutral Tannakian category. Then the group
$\cG=\Aut^\otimes(\cF)$ of tensor automorphisms of the fibre functor is
an affine algebraic complex group scheme, and the fibre functor $\cF$
determines an equivalence $(\cC,\cF)\cong(\Rep{\cG},\cF_\cG)$.
\end{theorem}

Suppose now $P\subsetneq G$ is a parabolic subgroup, as
in~\secref{sub:ind-red-homogeneous-vect}, and $(Q,\cK)$ is the quiver
with relations associated to $P$, as
in~\secref{sub:quiver-relation}. By
Theorem~\ref{thm:Tannaka-Grothendieck-Saavedra}, we can transfer the
structure of neutral Tannakian category from $\Rep{P}$ to
$\mod{Q,\cK}$ via the equivalence of
Theorem~\ref{thm:equiv-P-mod-QK-mod}, so $\mod{Q,\cK}$ becomes a
neutral Tannakian category such that the group of tensor automorphisms
of the fibre functor is isomorphic to $P$. Our next task is to
construct this structure of neutral Tannakian category on
$\mod{Q,\cK}$. Since $\mod{Q,\cK}$ is an abelian category, it suffices
to define an identity object, suitable dualization and tensor product
operations, and a fibre functor on $\mod{Q,\cK}$.

Since $\Rep{L}$ has identity object $\CC$ with the trivial $L$-action,
we define the identity object $\mathbbm{1}$ of $\mod{Q,\cK}$ as the
$Q$-module $(\bV,\bvarphi)$ given by $V_\lambda=0$ if
$\lambda\neq\lambda_0$, $V_{\lambda_0}=\CC$, $\varphi_a=0$, for all
$\lambda\in Q_0$, $a\in Q_1$, where $\lambda_0\in Q_0$ is the
isomorphism class of the trivial $L$-module
$M_{\lambda_0}\defeq\CC$. The space $\Hom(U,V)$ of linear maps
$f\colon U\to V$ between two $L$-representations $U$ and $V$ is
canonically an $L$-representation (this is the inner Hom in
$\Rep{L}$). In particular, the dual of $V$ is the $L$-representation
$V^*=\Hom(V,\CC)$. To construct the dualization functor on
$\mod{Q,\cK}$, we will use the involution
\[
Q_0\longrightarrow Q_0\colon \lambda\longmapsto\lambda^*,
\]
where $\lambda^*\in Q_0$ is the isomorphism class of the dual
irreducible $L$-module $M_\lambda^*$, for all $\lambda\in Q_0$, and
the 1-dimensional vector spaces
\begin{equation}\label{eq:C_lambda}
C_\lambda\defeq\Hom_L(M_\lambda,M_{\lambda^*}^*),
\end{equation}
for all $\lambda\in Q_0$. Note that $\dim C_\lambda=1$, by Schur's
Lemma, and we could have chosen the irreducible representations
$M_\lambda$, for all $\lambda\in Q_0$, so that either
$M_{\lambda^*}=M_\lambda^*$ or $M_\lambda=M_{\lambda^*}^*$ for all
$\lambda\in Q_0$, in which case $C_\lambda$ would be canonically
isomorphic to $\CC$. However, to make several natural isomorphisms
below become slightly clearer, we will refrain from making this choice
--- this has the effect that there is no \emph{canonical} isomorphism
$C_\lambda\cong\CC$, so the isotopical decomposition of
$M_{\lambda^*}^*$ becomes a canonical $L$-equivariant
isomorphism
\begin{equation}\label{eq:C_lambda-bis}
M_{\lambda^*}^*=C_\lambda\otimes M_\lambda.
\end{equation}
It follows now from~\eqref{eq:A-B-spaces.a}
and~\eqref{eq:C_lambda-bis} that there are canonical
isomorphisms
\begin{equation}\label{eq:A_dual}
A_{\lambda^*\mu^*}\cong C_\mu\otimes A_{\mu\lambda}\otimes C_\lambda^*,
\end{equation}
obtained by composing the transposition isomorphisms
$\Hom(M_{\mu^*},M_{\lambda^*})\cong\Hom(M_{\mu^*}^*,M_{\lambda^*}^*)$
with the isomorphisms (given by evaluation)
$M_{\lambda^*}^*=C_\lambda\otimes M_\lambda$ and
$M_{\mu^*}^*=C_\lambda\otimes M_\mu$. This isomorphism induces the
following linear isomorphisms for all $\lambda,\mu\in Q_0$ and any
(finite-dimensional) vector spaces $V_\lambda$ and $V_\mu$, with
$A_{\mu\lambda}$ given by~\eqref{eq:A-B-spaces.a}:
\begin{equation}\label{eq:isom-A-A-dual}\begin{split}
A_{\mu\lambda}\otimes\Hom(V_{\lambda^*}^*\otimes C_\lambda,V_{\mu^*}^*\otimes C_\mu)
&\cong (C_\mu\otimes A_{\mu\lambda}\otimes C_\lambda^*)
\otimes\Hom(V_{\lambda^*}^*,V_{\mu^*}^*)
\\&\cong
A_{\lambda^*\mu^*}\otimes\Hom(V_{\mu^*},V_{\lambda^*}),
\end{split}\end{equation}
where again we have used the transposition isomorphisms
$\Hom(V_{\lambda^*}^*,V_{\mu^*}^*)\cong\Hom(V_{\mu^*},V_{\lambda^*})$.

Let $(\bV,\bvarphi)$ be a $(Q,\cK)$-module. Then its \emph{dual $(Q,\cK)$-module} 
\begin{equation}\label{eq:dual-1}
\dual{(\bV,\bvarphi)}=(\dual{\bV},\dual{\bvarphi})
\end{equation}
is constructed as follows. For each $\lambda\in Q_0$,
\begin{equation}\label{eq:dual-2}
\dual{V}_\lambda\defeq V_{\lambda^*}^*\otimes C_\lambda.
\end{equation}
For each $\lambda,\mu\in Q_0$, let $\varphi_{\lambda^*\mu^*}\in
A_{\lambda^*\mu^*}\otimes\Hom(V_{\mu^*},V_{\lambda^*})$ be the vector
corresponding to the set
$\bvarphi_{\lambda^*\mu^*}=\{\varphi^{(i)}_{\lambda^*\mu^*}\}_{i=1}^{n_{\lambda^*\mu^*}}$
via the left-hand isomorphism~\eqref{eq:arrow-maps-equiv-maps}, with
$\varphi^{(i)}_{\lambda^*\mu^*}\colon V_{\mu^*}\to V_{\lambda^*}$ as
in~\eqref{eq:phi_a}. Let $\dual{\varphi}_{\mu\lambda}\in
A_{\lambda^*\mu^*}\otimes\Hom(\dual{V}_{\mu^*},\dual{V}_{\lambda^*})$
be the vector corresponding to
\begin{equation}\label{eq:negative-phi}
-\varphi_{\lambda^*\mu^*}\in A_{\lambda^*\mu^*}\otimes\Hom(V_{\mu^*},V_{\lambda^*}),
\end{equation}
via~\eqref{eq:isom-A-A-dual} (note the negative sign).
Then the set $\dual{\bvarphi}_{\mu\lambda}
=\{\duali{\varphi}{(i)}_{\mu\lambda}\}_{i=1}^{n_{\mu\lambda}}$ of
linear maps
\begin{equation}\label{eq:dual-3}
\duali{\varphi}{(i)}_{\mu\lambda}\colon\dual{V}_\lambda\longrightarrow\dual{V}_\mu
\end{equation}
corresponds to $\dual{\varphi}_{\mu\lambda}$ via the left-hand
isomorphism~\eqref{eq:arrow-maps-equiv-maps}, with $\bV$ replaced by
$\dual{\bV}$.

To construct a tensor product on $\mod{Q,\cK}$, we will use the
multiplicity vector spaces
\begin{equation}\label{eq:Littlewood-Richardson-decomposition-1}
C^\lambda_{\mu\nu}\defeq\Hom_L(M_\lambda,M_\mu\otimes M_{\nu}),
\end{equation}
for all $\lambda,\nu,\nu'\in Q_0$. In other words, the $L$-module
$M_\mu\otimes M_\nu$ has isotopical decomposition
\begin{equation}\label{eq:Littlewood-Richardson-decomposition-2}
M_\mu\otimes M_\nu=\bigoplus_{\lambda\in Q_0}C^\lambda_{\mu\nu}\otimes M_\lambda.
\end{equation}
Let $(\bV,\bvarphi)$ and $(\bV',\bvarphi')$ be two $(Q,\cK)$-modules. Then
their tensor product
\begin{equation}\label{eq:tensor-product-1}
(\bU,\bpsi)=(\bV\otimes\bV',\bvarphi\otimes\bvarphi')\defeq(\bV,\bvarphi)\otimes(\bV',\bvarphi')
\end{equation}
is constructed as follows. For each $\lambda\in Q_0$,
\begin{equation}\label{eq:tensor-product-2}
U_\lambda=(V\otimes V')_\lambda
\defeq\bigoplus_{\mu,\nu\in Q_0}V_\mu\otimes V'_{\nu}\otimes C_{\mu\nu}^\lambda.
\end{equation}
For each $\alpha,\beta,\alpha',\beta'\in Q_0$, the sets
$\bvarphi_{\beta\alpha}=\{\varphi_{\beta\alpha}^{(i)}\}_{i=1}^{n_{\beta\alpha}}$
and
$\bvarphi'_{\beta'\alpha'}=\{\varphi_{\beta'\alpha'}^{\prime(i)}\}_{i=1}^{n_{\beta'\alpha'}}$,
correspond via~\eqref{eq:arrow-maps-equiv-maps} to $L$-equivariant
linear maps
\begin{gather*}
\varphi_{\beta\alpha}\colon\glu\otimes V_\alpha\otimes
M_\alpha\longrightarrow V_\beta\otimes M_\beta,
\quad 
\varphi'_{\beta'\alpha'}\colon\glu\otimes V'_{\alpha'}\otimes
M_{\alpha'}\longrightarrow V'_{\beta'}\otimes M_{\beta'},
\end{gather*}
respectively, that induce other ones given by 
\begin{gather*}
\varphi_{\beta\alpha,\alpha'}\defeq\varphi_{\beta\alpha}\otimes\Id_{V'_{\alpha'}\otimes M_{\alpha'}}\colon
\glu\otimes V_\alpha\otimes M_\alpha
\otimes V'_{\alpha'}\otimes M_{\alpha'}
\longrightarrow V_\beta\otimes M_\beta
\otimes V'_{\alpha'}\otimes M_{\alpha'}.
\\
\varphi'_{\beta'\alpha',\alpha}\defeq\Id_{V_\alpha\otimes M_\alpha}\otimes
\varphi'_{\beta'\alpha'}\colon
\glu\otimes 
V_\alpha\otimes M_\alpha\otimes
V'_{\alpha'}\otimes M_{\alpha'}\longrightarrow 
V_\alpha\otimes M_\alpha\otimes V'_{\beta'}\otimes M_{\beta'},
\end{gather*}
respectively obtained tensoring by $V'_{\alpha'}\otimes M_{\alpha'}$
and $V_{\alpha}\otimes M_{\alpha}$,
where~\eqref{eq:Littlewood-Richardson-decomposition-2} implies
\begin{gather*}
\glu\otimes V_\alpha\otimes M_\alpha
\otimes V'_{\alpha'}\otimes M_{\alpha'}
=\bigoplus_{\lambda\in Q_0}\glu\otimes V_\alpha\otimes V'_{\alpha'}\otimes C^\lambda_{\alpha\alpha'}
\otimes M_\lambda,\\
V_\beta\otimes M_\beta
\otimes V'_{\alpha'}\otimes M_{\alpha'}
=\bigoplus_{\mu\in Q_0}V_\beta\otimes V'_{\alpha'}\otimes C^{\mu}_{\beta\alpha'}
\otimes M_{\mu},\\
V_\alpha\otimes M_\alpha\otimes V'_{\beta'}\otimes M_{\beta'}
=\bigoplus_{\mu\in Q_0}V_\alpha\otimes V'_{\beta'}\otimes C^{\mu}_{\alpha\beta'}
\otimes M_{\mu}.
\end{gather*}
Hence $\varphi_{\beta\alpha,\alpha'}$ and $\varphi'_{\beta'\alpha',\alpha}$
admit decompositions
\[
\varphi_{\beta\alpha,\alpha'}=\sum_{\lambda,\mu\in Q_0}\varphi_{\beta\alpha,\alpha'}^{\mu\lambda},
\quad
\varphi'_{\beta'\alpha',\alpha}=\sum_{\lambda,\mu\in Q_0}\varphi_{\beta'\alpha',\alpha}^{\prime\mu\lambda},
\]
where
\begin{gather*}
\varphi_{\beta\alpha,\alpha'}^{\mu\lambda}\colon 
\glu\otimes V_\alpha\otimes V'_{\alpha'}\otimes C^\lambda_{\alpha\alpha'}
\otimes M_\lambda
\longrightarrow
V_\beta\otimes V'_{\alpha'}\otimes C^{\mu}_{\beta\alpha'}
\otimes M_{\mu},
\\
\varphi_{\beta'\alpha',\alpha}^{\prime\mu\lambda}\colon 
\glu\otimes V_\alpha\otimes V'_{\alpha'}\otimes C^\lambda_{\alpha\alpha'}
\otimes M_\lambda
\longrightarrow
V_\alpha\otimes V'_{\beta'}\otimes C^{\mu}_{\alpha\beta'}\otimes M_{\mu}.
\end{gather*}
Since $\varphi_{\beta\alpha,\alpha'}^{\mu\lambda}$ and
$\varphi_{\beta'\alpha',\alpha}^{\prime\mu\lambda}$ are $L$-equivariant
maps, they may be viewed as linear maps
\begin{gather*}
\varphi_{\beta\alpha,\alpha'}^{\mu\lambda}\colon 
V_\alpha\otimes V'_{\alpha'}\otimes C^\lambda_{\alpha\alpha'}
\longrightarrow 
A_{\mu\lambda}\otimes V_\beta\otimes V'_{\alpha'}\otimes C^{\mu}_{\beta\alpha'}.
\\
\varphi_{\beta'\alpha',\alpha}^{\prime\mu\lambda}\colon 
V_\alpha\otimes V'_{\alpha'}\otimes C^\lambda_{\alpha\alpha'}
\longrightarrow
A_{\mu\lambda}\otimes V_\alpha\otimes V'_{\beta'}\otimes C^{\mu}_{\alpha\beta'}, 
\end{gather*}
via~\eqref{eq:arrow-maps-equiv-maps}. Adding them together for all
$\alpha,\alpha',\beta,\beta'\in Q_0$, we obtain another linear map 
\begin{gather}\label{eq:psi-tensor_product}
\begin{split}
\psi_{\mu\lambda}=
\sum_{\alpha,\alpha',\beta\in
  Q_0}\varphi_{\beta\alpha,\alpha'}^{\mu\lambda}
+\sum_{\alpha,\alpha',\beta'\in
  Q_0}
\varphi_{\beta'\alpha',\alpha}^{\prime\mu\lambda}
&\colon 
\\
\bigoplus_{\alpha,\alpha'\in Q_0}
V_\alpha \otimes V'_{\alpha'}\otimes &C^\lambda_{\alpha\alpha'}
\longrightarrow
A_{\mu\lambda}\otimes\bigoplus_{\beta,\beta'\in Q_0}V_\beta\otimes
V'_{\beta'}\otimes C^{\mu}_{\beta\beta'},
\end{split}\end{gather}
where 
\begin{gather*}
U_\lambda=\bigoplus_{\alpha,\alpha'\in Q_0}V_\alpha\otimes V'_{\alpha'}\otimes C^\lambda_{\alpha\alpha'},
\quad 
U_{\mu}=\bigoplus_{\beta,\beta'\in Q_0}V_\beta\otimes V'_{\beta'}\otimes C^{\mu}_{\beta\beta'}
\end{gather*}
and hence $\psi_{\mu\lambda}\colon U_\lambda\to A_{\mu\lambda}\otimes
U_{\mu}$. Then the set $\bpsi_{\mu\lambda}
=\{\psi_{\mu\lambda}^{(i)}\}_{i=1}^{n_{\mu\lambda}}$ of linear maps
\begin{equation}\label{eq:tensor-product-3}
\psi_{\mu\lambda}^{(i)}=(\varphi\otimes\varphi')_{\mu\lambda}^{(i)}
\colon U_\lambda\longrightarrow U_{\mu}
\end{equation}
corresponds to $\psi_{\mu\lambda}$ via the left-hand
isomorphism~\eqref{eq:arrow-maps-equiv-maps}, with $\bV$ replaced by
$\bU$.

The category $\mod{Q,\cK}$ is also equipped with a fibre functor 
\begin{equation}\label{eq:fibre-functor}
\cF_{Q,\cK}\colon\mod{Q,\cK}\longrightarrow\VectC,
\quad (\bV,\bvarphi)\longmapsto\bigoplus_{\lambda\in Q_0}V_\lambda\otimes M_\lambda.
\end{equation}

\begin{lemma}\label{lem1}
The pair $(\mod{Q,\cK},\cF_{Q,\cK})$ is a neutral Tannakian category,
and there is an equivalence of neutral Tannakian categories
\begin{equation}\label{eq:lem1}
(\mod{Q,\cK},\cF_{Q,\cK})\cong(\Rep{P},\cF_P). 
\end{equation}
Therefore the affine algebraic complex group scheme
$\Aut^\otimes(\cF_{Q,\cK})$ is identified with $P$.
\end{lemma}

\begin{proof}
In view of Theorem~\ref{thm:equiv-P-mod-QK-mod}, we need to check that
the equivalence~\eqref{eq:equiv-P-mod-QK-mod} identifies the
dualization operation, the tensor product operation, and the fibre
functor on the categories $\mod{Q,\cK}$ and $\Rep{P}$. 
To check that the dualization operations coincide, we fix a
$(Q,\cK)$-module $(\bV,\bvarphi)$, with dual $\dual{(\bV,\bvarphi)}
=(\dual{\bV},\dual{\bvarphi})$. Let $\VV$ and $\dual{\VV}$ be the
$L$-modules associated to $\bV$ and $\dual{\bV}$ (as
in~\eqref{eq:isotopical-dec-V}), and $\rho\colon P\to\GL(\VV)$ the
$P$-module corresponding to
$(\bV,\bvarphi)$. Then~\eqref{eq:C_lambda-bis} and~\eqref{eq:dual-2} give
canonical $L$-equivariant isomorphisms
\[
\dual{V}_\lambda\otimes M_\lambda\cong V^*_{\lambda^*}\otimes M_{\lambda^*}^*,
\]
for all $\lambda\in Q_0$, so the $L$-modules $\VV$ and $\dual{\VV}$
are related by another canonical isomorphism
\[
\dual{\VV}=\bigoplus_{\lambda\in Q_0}\dual{V}_\lambda\otimes
M_\lambda \cong \bigoplus_{\lambda\in Q_0}V^*_{\lambda^*}\otimes
M_{\lambda^*}^* =\VV^*.
\]
Now, the $P$-module $\rho\colon P\to\GL(\VV)$ determines an object
$(\VV,\tau)$ of $\Rep{L,\glu}$ (see~\eqref{eq:intermediate-equiv}),
where the Lie-algebra representation $\tau\colon\glu\to\End\VV$ is
given by $\tau=d\rho|_{\glu}$ (see~\eqref{eq:def-tau}), whereas its
dual $P$-module $\dual{\rho}\colon P\to\GL(\VV^*)$ determines the dual
object $(\VV^*,\dual{\tau})$ of $\Rep{L,\glu}$, where the dual
Lie-algebra representation $\dual{\tau}\colon\glu\to\End\VV^*$ is
given by
\begin{equation}\label{eq:hat-tau}
\dual{\tau}(e)=-\tau(e)^*
\end{equation}
for all $e\in\glu$ ($\tau(e)^*\colon\VV^*\to\VV^*$ being the dual map
of $\tau(e)\colon\VV\to\VV$). Comparing~\eqref{eq:hat-tau}
with~\eqref{eq:negative-phi}, it now follows from the construction of
the maps~\eqref{eq:dual-3} that the object $(\VV^*,\dual{\tau})$ of
$\Rep{L,\glu}$ corresponds to the $Q$-module
$\dual{(\bV,\bvarphi)}$. Note that the $Q$-module $\dual{(\bV,\bvarphi)}$
satisfies the set of relations $\cK$, because it corresponds to the
$P$-module $\dual{\rho}\colon P\to\GL(\VV^*)$.

To check that the tensor product operations coincide, we fix
$(Q,\cK)$-modules $(\bV,\bvarphi)$ and $(\bV',\bvarphi')$. Let
$(\bU,\bpsi)=(\bV,\bvarphi)\otimes(\bV',\bvarphi')$ be their tensor product,
so $\bU=\bV\otimes\bV'$ is given
by~\eqref{eq:tensor-product-2}. Then~\eqref{eq:Littlewood-Richardson-decomposition-2}
implies that the $L$-modules $\VV,\VV'$ and $\UU$ associated (as
in~\eqref{eq:isotopical-dec-V}) to the $Q_0$-graded vector spaces
$\bV,\bV'$ and $\bU=\bV\otimes\bV'$, respectively, satisfy
\[
\UU=\VV\otimes\VV'.
\]
Let $\rho\colon P\to\GL(\VV)$ and $\rho'\colon P\to\GL(\VV')$ be the
$P$-modules corresponding to the $(Q,\cK)$-modules $(\bV,\bvarphi)$
and $(\bV',\bvarphi')$, respectively. They determine objects
$(\VV,\tau)$ and $(\VV',\tau')$ of $\Rep{L,\glu}$
(see~\eqref{eq:intermediate-equiv}), with $\tau=d\rho|_{\glu},
\tau=d\rho'|_{\glu}$ (see~\eqref{eq:def-tau}), and their tensor
product $\rho''\defeq\rho\otimes\rho'\colon P\to\GL(\UU)$ determines
another one $(\UU,\tau'')$, with $\tau''=d\rho''|_{\glu}$ given by
\begin{equation}\label{eq:tensor-tau}
\tau''(e)=\tau(e)\otimes\Id_{\VV'}+\Id_{\VV}\otimes\tau(e), 
\end{equation}
for all $e\in\glu$. Comparing~\eqref{eq:tensor-tau}
with~\eqref{eq:psi-tensor_product}, it now follows from the
construction of the maps~\eqref{eq:tensor-product-3} that the object
$(\UU,\tau'')$ of $\Rep{L,\glu}$ corresponds to the $Q$-module
$(\bU,\bpsi)$. The fact that the $Q$-module $(\bU,\bpsi)$ satisfies the
set of relations $\cK$ follows because it corresponds to the tensor
product $P$-module $\rho''\colon P\to\GL(\UU)$.

Finally, the identification of the fibre functors follows by
comparing~\eqref{eq:isotopical-dec-V} and \eqref{eq:fibre-functor}.
\end{proof}

Note that the \emph{inner} Hom spaces of the rigid tensor categories
$\Rep{P}$ and $\mod{Q,\cK}$, corresponding to each other via the
equivalence~\eqref{eq:equiv-P-mod-QK-mod}, are respectively given by
\[
\Hom(\VV,\VV')\cong\VV'\otimes\VV^*, \quad 
\Hom((\bV,\bvarphi),(\bV',\bvarphi'))=(\bV',\bvarphi')\otimes\dual{(\bV,\bvarphi)}.
\]

We are now ready to prove the main result of this section.

\begin{proposition}\label{prop:lem1}
There is an equivalence of neutral Tannakian categories between 
the category of holomorphic homogeneous vector bundles on $G/P$ and
the category of $(Q,\cK)$-modules with the fibre functor $\cF_{Q,\cK}$.
\end{proposition}

\begin{proof}
This follows from Lemmas~\ref{lem:ind-red-homgeneous-vect}
and~\ref{lem1}, since the equivalence in
Lemma~\ref{lem:ind-red-homgeneous-vect} is an equivalence of neutral
Tannakian categories. In the category of holomorphic homogeneous
vector bundles $E$ over $G/P$, the identity object is the trivial
homogeneous vector bundle $\CC_{G/P}\defeq G\times^P\CC=G/P\times\CC$,
and the fibre functor $\cF_{G/P}$ is given by
\[
\cF_{G/P}(E)=E_o, 
\]
where $o=P\in G/P$ is the base point of $G/P$, with isotropy group
$P$.
\end{proof}

\section{Equivariant principal bundles and torsors in quiver bundles}
\label{sec:equivpfb-Qbun}

Building on the equivalences of Lemma~\ref{lem1} and
Proposition~\ref{prop:lem1}, in this section we adapt Nori's functor
approach to principal bundles, to obtain a Tannakian description of
equivariant principal bundles over $X\times G/P$. This view point will
be useful, in~\secref{sec:dim-red}, in the understanding via
dimensional reduction to $X$ of the algebro-geometric
(semi/poly)stability condition for equivariant holomorphic principal
bundles on $X\times G/P$.

\subsection{Notation}
\label{sub:preliminaries}

Throughout this paper, $X$ is a smooth irreducible complex projective
variety, $H$ is a complex connected reductive affine algebraic group,
and $\Rep{H}$ is the category of $H$-modules, that is, (holomorphic
finite-dimensional complex) representations of $H$ and their
$H$-equivariant linear maps. As in~\secref{sec2}, $G$ is a connected
simply connected semisimple complex affine algebraic group, and
$P\subsetneq G$ is a parabolic subgroup, with unipotent radical
$R_u(P)\subset P$. The Lie algebras of $R_u(P)\subset P\subset G$ are
denoted
\[
\glu\subset\glp\subset\glg,
\]
respectively. We also fix a Levi subgroup
\[
L=L(P)\subset P
\]
of $P$, that is, a maximal connected reductive subgroup of $P$. Recall
that any two Levi subgroups are conjugate by some element of
$R_u(P)\subset P$ \cite[p. 185, Theorem]{Hu}. 

Given a complex manifold $M$ with a holomorphic left $G$-action, a
\emph{$G$-equivariant holomorphic principal $H$-bundle} over $M$ is a
holomorphic principal $H$-bundle $\pi:E'_H\to M$ on $M$, together with
a holomorphic left $G$-action $\rho:G\times E'_H\to E'_H$ on the total
space of $E'_H$ that commutes with $\pi$, and such that for all
$(g,p)\in G\times M$, the map $\rho_{g,p}: E'_p\to E'_{g\cdot p}$
induced by this action is an isomorphism of $H$-manifolds, where
$E'_p=\pi^{-1}(p)$.
In this paper, we are concerned with $G$-equivariant holomorphic
principal $H$-bundles on the $G$-manifold $M=X\times G/P$. The group
$G$ acts on $X\times G/P$ by the trivial action on $X$ and the
left-translation action induced by left multiplication on the flag
variety $G/P$. If $X$ is a point, we recover the notion of
\emph{homogeneous holomorphic principal $H$-bundle} over
$G/P$~\cite{Bi}, that is, a holomorphic principal $H$-bundle $E'_H\to
G/P$, together with a holomorphic left $G$-action on $E'_H$ that lifts
the left-translation action of $G$ on $G/P$. If, furthermore,
$H=\GL(r,\CC)$ for some integer $r>0$, then $E'_H$ is equivalent to a
rank-$r$ homogeneous holomorphic vector bundle over $G/P$~\cite{AG2},
that is, a $G$-equivariant holomorphic vector bundle $E'$ over $G/P$
(see~\secref{sub:ind-red-homogeneous-vect}).

\subsection{Induction and reduction of equivariant principal bundles}
\label{sub:Induction-Reduction}

We start constructing an equivalence betweeen the groupoid of
$G$-equivariant holomorphic principal $H$-bundles on $X\times G/P$,
and the groupoid of $P$-equivariant holomorphic principal $H$-bundles
on the $P$-variety $X$, with $P$ acting trivially on $X$ (as usual, a
groupoid is a category in which every arrow is invertible). The
morphisms of the former groupoid are $G$-equivariant isomorphisms of
holomorphic principal $H$-bundles. The morphisms of the latter
groupoid are $P$-equivariant isomorphisms of holomorphic principal
$H$-bundles.

\begin{lemma}\label{lem:induction-reduction}
There is an equivalence between the groupoid of $G$-equivariant
holomorphic principal $H$-bundles on $X\times G/P$, and the groupoid
of $P$-equivariant holomorphic principal $H$-bundles on $X$ .
\end{lemma}

\begin{proof}
For any $G$-equivariant holomorphic principal $H$-bundle $E_H'$ on
$X\times G/P$, the restriction of $E_H'$ to the slice $X\cong X\times
P/P\subset X\times G/P$ is a $P$-equivariant holomorphic principal
$H$-bundle on $X$. Conversely, given a $P$-equivariant holomorphic
principal $H$-bundle $E_H$ on $X$, consider the pullback $p^*_1 E_H\to
X\times G$, where $p_1\colon X\times G\to X$ is the canonical
projection. The left-translation action of $G$ on itself and the
trivial action of $G$ on $X$ together define an action of $G$ on
$X\times G$. Since $p^*_1 E_H$ is a pullback from $X$, this action of
$G$ on $X\times G$ has a canonical lift to a left action of $G$ on
$p^*_1 E_H$.

The right-translation action of $P$ on $G$ and the trivial action of
$P$ on $X$ together produce an action of $P$ on $X\times G$. Using the
action of $P$ on $E_H$, we get a lift of this action of $P$ on
$X\times G$ to an action of $P$ on $p^*_1 E_H$. The quotient $(p^*_1
E_H)/P$ for this action of $P$ on $p^*_1 E_H$ is a principal
$H$-bundle on $X\times G/P$.  The left action of $G$ on $p^*_1 E_H$
constructed earlier descends to an action of $G$ on the principal
$H$-bundle $(p^*_1 E_H)/P\to X\times G/P$ (this is because the
left-translation action of $G$ on itself commutes with the
right-translation action of $P$ on $G$).
\end{proof}

In the situation of the above lemma, we say a $G$-equivariant
holomorphic principal $H$-bundle $E_H'$ on $X\times G/P$ is
\emph{induced} by the corresponding $P$-equivariant holomorphic
principal $H$-bundle $E_H$ on $X$, denoted
\begin{equation}\label{eq:induced-pfb}
E_H'=G\times^PE_H\longrightarrow X\times G/P.
\end{equation}

According to Lemma~\ref{lem:induction-reduction}, our next task is to
give an appropriate description of $P$-equivariant holomorphic
principal $H$-bundles over $X$. Since we will apply Tannakian methods
in this paper, this will be achieved using to the category of
$P$-equivariant holomorphic vector bundles over $X$. The latter
category has been described in a combinatorial way using
representations of a quiver with relations~\cite{AG2}, as reviewed
in~\secref{sub:rep-quiver-relation}.

\subsection{Equivariant vector bundles  and quiver bundles}
\label{sub:Q-bundles}

\newcommand{\Vect}[1]{\operatorname{Vect}_X(#1)}
\newcommand{\VectX}{\operatorname{Vect}_X}
\newcommand{\VectXP}{\operatorname{Vect}_{X\times P}}
\newcommand{\VectG}{\operatorname{Vect}_{X\times G/P}^G}
\newcommand{\VectP}{\operatorname{Vect}_X^P}
\newcommand{\HBunG}{H\textnormal{-}\operatorname{Bun}_{X\times G/P}^G}
\newcommand{\HBunP}{H\textnormal{-}\operatorname{Bun}_X^P}

Let $Q$ be a quiver. Replacing vector spaces by holomorphic vector
bundles in the definition of $Q$-module given 
in~\secref{sub:rep-quiver-relation}, we obtain the category $\Vect{Q}$
of holomorphic quiver bundles. More precisely,
\begin{itemize}
\item 
a \emph{holomorphic
  $Q$-bundle} on $X$ is a pair $(\bE,\bphi)$ consisting of a set $\bE$
of holomorphic vector bundles $E_v$ on $X$, indexed by the vertices
$v\in Q_0$, and a set $\bphi$ of holomorphic vector-bundle
homomorphisms $\phi_a\colon E_{ta}\to E_{ha}$, indexed by the arrows $a\in Q_1$.
\item
a \emph{homomorphism} $\bf\colon(\bE,\bphi)\to(\bE',\bphi')$ of holomorphic $Q$-bundles is a
set of homomorphisms $f_v\colon E_v\to E_v'$ of holomorphic vector
bundles, indexed by the vertices $v\in Q_0$, such that $\phi_a'\circ
f_{ta}=f_{ha}\circ\phi_a$ for all $a\in Q_1$. 
\end{itemize}
Let $\cK$ be a set of relations of $Q$. In this section, we will be
interested in the full subcategory
\[
\Vect{Q,\cK}\subset\Vect{Q}
\]
of \emph{holomorphic $(Q,\cK)$-bundles} on $X$, that is, holomorphic
$Q$-bundles that satisfy the set of relations $\cK$, where
$(\bE,\bphi)$ \emph{satisfies a relation~\eqref{eq:relation}} if
$\phi(r)\defeq c_1\phi(p_1) +\cdots+ c_k\phi(p_k)$ is zero, with a
path $p$ of $Q$ determining a holomorphic vector-bundle homomorphism
$\phi(p)\colon E_{tp}\to E_{hp}$, given by
$\phi(p)=\phi_{a_\ell}\circ\cdots\circ\phi_{a_1}$ for a non-trivial
path~\eqref{eq:path}, and by the identity $\phi(p)=\Id_{E_v}\colon
E_v\to E_v$ for the trivial path $p=e_v$ at any $v\in Q_0$.

Now, let $P\subsetneq G$ be a parabolic subgroup, and $(Q,\cK)$ the
quiver with relations associated to $P$ (see
Definition~\ref{def:quiver-relations}). Then the category
$\Vect{Q,\cK}$ of holomorphic $(Q,\cK)$-bundles has an identity
object, a dualization functor and a tensor product bifunctor, denoted
\begin{gather*}
\dual{(-)}\colon\Vect{Q,\cK}^{\operatorname{op}}\longrightarrow\Vect{Q,\cK},\\
\otimes\colon\Vect{Q,\cK}\times\Vect{Q,\cK}\longrightarrow\Vect{Q,\cK},
\end{gather*}
respectively, constructed exactly as in~\secref{sub:Tannakian-QK-mod},
replacing vector spaces by holomorphic vector bundles
(see~\eqref{eq:dual-1} and~\eqref{eq:tensor-product-1}). 

\begin{theorem}\label{prop:cat-Q-bundles}
There are equivalences of categories, 
preserving the identity objects, dualization and tensor product,
between
\begin{itemize}
\item the category 
of $G$-equivariant holomorphic vector bundles on $X\times G/P$,
\item the category 
of $P$-equivariant holomorphic vector bundles on $X$,
\item the category 
of holomorphic $(Q,\cK)$-bundles on $X$.
\end{itemize}
\end{theorem}

\begin{proof}
This follows as in the part of the proof of
Lemma~\ref{ex:Tannakian-Rep} and Proposition~\ref{prop:cat-Q-bundles}
concerning the identity objects, the dualization and the tensor product,
with~\cite[Lemma~2.8 and Theorem 2.5]{AG2} playing the roles of
Lemma~\ref{lem:ind-red-homgeneous-vect} and
Theorem~\ref{thm:equiv-P-mod-QK-mod}, respectively.
\end{proof}

\subsection{Equivariant principal bundles as torsors in quiver bundles}
\label{sub:Tannaka-equiv-bundles}

Let $(Q,\cK)$ be the quiver with relations associated to the parabolic
subgroup $P\subsetneq G$. In the following definition (cf.,
e.g.,~\cite[p. 82]{Si}), the adjectives ``strict exact faithful
tensor'' applied to a functor mean the four abstract conditions
$\rm{F}_1$--$\rm{F}_4$ listed by Nori~\cite[p. 77]{No} (in particular,
the functor must be compatible with the operations of taking dual and
tensor product).

\begin{definition}\label{def:torsor}
An \emph{$H$-torsor in holomorphic $(Q,\cK)$-bundles over $X$} is a 
strict exact faithful tensor functor
\begin{equation}\label{eq:thm1}
\FF\colon\Rep{H}\longrightarrow\Vect{Q,\cK}.
\end{equation}
\end{definition}

\begin{theorem}\label{thm1}
There is a bijective correspondence between the $P$-equivariant
holomorphic principal $H$-bundles on $X$ and the $H$-torsors in
holomorphic $(Q,\cK)$-bundles over $X$.
\end{theorem}

\begin{proof}
Let $E_H$ be a $P$-equivariant holomorphic principal $H$-bundle over
$X$. For each $H$-module $W\in\Rep{H}$, we have the associated
holomorphic vector bundle
\[
E_W=E_H(W)\defeq E_H\times^H W \longrightarrow X.
\]
The action of $P$ on $E_W$ and the trivial action of $P$ on $W$
together produce an action of $P$ on $E_H\times W$. This action
descends to an action of $P$ on $E_W$. Now by
Theorem~\ref{prop:cat-Q-bundles}, the $P$-equivariant holomorphic
vector bundle $E_W$ produces an object $\FF(W)$ of $\Vect{Q,\cK}$.

For the converse direction, let $\FF$ be an $H$-torsor in holomorphic
$(Q,\cK)$-bundles over $X$, i.e. a strict exact faithful tensor
functor, as in~\eqref{eq:thm1}. Let $\VectX$ be the category of
holomorphic vector bundles on $X$. Define the functor 
\begin{equation}\label{eq:functor-S}
S\colon\Vect{Q,\cK}\longrightarrow\VectX, \quad
(\bE,\bphi)\longmapsto\bigoplus_{\lambda\in Q_0}E_\lambda\otimes M_\lambda.
\end{equation}
Then the functor $S\circ\FF\colon\Rep{H}\to\VectX$ produces an
algebraic principal $H$-bundle over $X$
\cite[Lemma~2.3~and~Proposition~2.4]{No}. This principal $H$-bundle
over $X$ will be denoted by $E_H(\FF)$.
Our aim is to construct an action of $P$ on the principal $H$-bundle
$E_H(\FF)$.
Let
\[
p_X\colon X\times P\longrightarrow X
\]
be the canonical projection onto the first factor. Let $\VectXP$ be
the category of holomorphic vector bundles on $X\times P$. Consider
the functor
\begin{equation}\label{c1}
p^*_X\circ S\circ\FF\colon\Rep{H}\longrightarrow\VectXP
\end{equation}
that maps any $W\in\Rep{H}$ into the pulled back vector bundle $p^*_X
(S\circ\FF(W))$. As above, this functor defines a principal $H$-bundle
on $X\times P$ \cite[Lemma~2.3~and~Proposition~2.4]{No}; this
principal $H$-bundle on $X\times P$ will be denoted by
$\widetilde{E}_H(\FF)$. The principal $H$-bundle
$\widetilde{E}_H(\FF)$ is evidently identified with the pullback
$p^*_XE_H(\FF)$.

For any $W\in\Rep{H}$, the vector bundle $p^*_X(S\circ\FF(W))$ has a
natural automorphism
\begin{equation}\label{c2}
A_W\colon p^*_X (S\circ\FF(W))\longrightarrow p^*_X (S\circ\FF(W)), 
\end{equation}
which is constructed as follows: by Theorem~\ref{prop:cat-Q-bundles},
the vector bundle $S\circ\FF(W)$ on $X$ is equipped with an action of
$P$; this action of $P$ on $S\circ\FF(W)$ produces an automorphism of
the pullback $p^*_X(S\circ\FF(W))$. More precisely, the automorphism
sends any $w\in(p^*_X(S\circ\FF(W)))_{(x,g)}=(S\circ\FF(W))_x$ to the
image of $w$ under the action of $g\in P$.

Now we have a functor
\[
A\colon\Rep{H}\longrightarrow\VectXP
\]
that sends any $W\in\Rep{H}$ to $p^*_X(S\circ\FF(W))$, and sends any
homomorphism of $H$-modules
\[
\rho\colon W_1\longrightarrow W_2
\]
to $A_{W_2}\circ \rho'\circ A^{-1}_{W_1}$, where $\rho'\colon p^*_X
(S\circ\FF(W_1)) \longrightarrow p^*_X (S\circ\FF(W_2))$ is the image
of $\rho$ by the functor $p^*_X\circ S\circ\FF$.

The functor $A$ defined above produces an automorphism of the
principal $H$-bundle $p^*_XE_H(\FF)$. This automorphism defines a
morphism
\[
\psi\colon P\times E_H(\FF)\longrightarrow E_H(\FF),
\]
which commutes with the right action of $H$. To show that $\psi$ is an
action of $P$, take a faithful $H$-module $W$. Consider the principal
$\text{GL}(W)$-bundle
$E_{\text{GL}(W)}(\FF)=E_H(\FF)\times^H\text{GL}(W)$ over $X$ obtained
by extending the structure group of $E_H(\FF)$ using the homomorphism
$\rho_W\colon H\longrightarrow\text{GL}(W)$ given by the action of $H$
on $W$. Since $E_H(\FF)$ is embedded in the total space
$E_{\text{GL}(W)}(\FF)$ (as $\rho_W$ is injective), and $\psi$ is the
restriction of a $P$-action on this principal $\text{GL}(W)$-bundle
$E_{\text{GL}(W)}(\FF)$, we conclude that $\psi$ defines a $P$-action
on $E_H(\FF)$.
\end{proof}

Combining Theorem~\ref{thm1} with
Lemma~\ref{lem:induction-reduction}, we obtain the following.

\begin{corollary}\label{cor1}
There is a bijective correspondence between the $G$-equivariant
principal $H$-bundles on $X\times G/P$ and the $H$-torsors in
holomorphic $(Q,\cK)$-bundles over $X$.
\end{corollary}

\section{Dimensional reduction for equivariant principal bundles}
\label{sec:dim-red}

In this section, we study algebro-geometric (semi/poly)stability
conditions for equivariant holomorphic principal $H$-bundles on
$X\times G/P$, their dimensional reduction to $X$, via the
correspondences of Lemma~\ref{lem:induction-reduction},
Theorem~\ref{thm1} and Corollary~\ref{cor1}, and the corresponding
gauge-theoretic equations. Throughout this section, we use the
notation of~\secref{sub:preliminaries}, so $X$ is a smooth irreducible
complex projective variety, $H$ is a complex connected reductive
affine algebraic group, $G$ is a connected simply connected semisimple
complex affine algebraic group, $P\subsetneq G$ is a parabolic
subgroup, and $L\subset P$ is a Levi subgroup. The Lie algebras of
$L\subset P\subset G$ are denoted $\gll\subset\glp\subset\glg$,
respectively, and $(Q,\cK)$ is the quiver with relations associated to
$P$.

\subsection{Semistable and polystable $G$-equivariant principal
bundles on $X\times G/P$}
\label{sub:ss-G-equiv}

Choose a Cartan subalgebra $\glt\subset\gll\subset\glg$, and a system
$\cS$ of simple roots of $\glg$ with respect to $\glt$, such that all
the negative roots of $\glg$ with respect to $\glt$ are roots of
$\glp$. Let $\Delta_+(\glr)$ be the set of positive roots of $\glg$
with respect to $\glt$ and $\cS$ that are not roots of $\gll$ (this
notation is motivated in~\cite[\S 4.2.1]{AG2}), and $\Sigma\subset\cS$
the set of simple roots of $\glg$ that are not roots of $\glp$. Let
$K\subset G$ be a maximal compact Lie subgroup. Then the $K$-invariant
K\"ahler 2-forms on the complex $G$-manifold $K/(K\cap P)\cong G/P$
are parametrized by elements
$\bepsilon=\{\epsilon_\alpha\}_{\alpha\in\Sigma}$ of
$\RR_{>0}^\Sigma$, where $\RR_{>0}$ and $\RR_{>0}^\Sigma$ denote the
set of positive real numbers and the set of maps
\begin{equation}\label{eq:bepsilon}
\bepsilon\colon\Sigma\longrightarrow\RR_{>0},\,
\alpha\longmapsto\epsilon_\alpha,
\end{equation}
respectively (see~\cite[p. 38, Lemma 4.8]{AG2}). The K\"ahler form on
$G/P$ associated to $\bepsilon$ will be denoted $\omega_{\bepsilon}$.
We equip the complex manifolds $X$, $G/P$ and $X\times G/P$ with fixed
K\"ahler 2-forms $\omega$, $\omega_\bepsilon$ and $p^*_1\omega
+p^*_2\omega_\bepsilon$, respectively, where $p_i$ is the projection
of $X\times G/P$ to the $i$-th factor.
We define the degree of a holomorphic vector bundle $E$ on $X$ by the
formula
\begin{equation}\label{eq:deg}
\deg E=\frac{2\pi}{(n-1)!\Vol(X)}\int_Xc_1(E)\wedge\omega^{n-1},
\end{equation}
where $n=\dim_\CC X$ and $\Vol(X)=\int_X \omega^n/n!$ is the volume of
$X$. We use this normalization convention also for the degree of
torsion-free coherent sheaves on $X$, and for the degree of
torsion-free coherent sheaves on $X\times G/P$ with respect to the
K\"ahler form $p^*_1\omega +p^*_2\omega_\bepsilon$.

We now recall the definition of a semistable principal $H$-bundle (see
\cite{Ra,RR,AB}). Let $Z_0(H)$ be the connected component of the
center of $H$ containing the identity element.

\begin{definition}\label{ds1}
A holomorphic principal $H$-bundle $E_H$ on $X\times G/P$
is called \emph{semistable} (respectively, \emph{stable}) if for every
quadruple $(B,U,E_B,\chi)$, where
\begin{itemize}
\item $B$ is a proper parabolic subgroup of $H$,
\item $\iota\colon U\hookrightarrow X\times G/P$ is a non-empty Zariski
open subset such that the complement
$U^c$ is of complex codimension at least two,
\item $E_B\subset E_H\vert_U$ is a holomorphic reduction of structure
group to $B$ over the open subset $U$, and
\item $\chi$ is a strictly anti-dominant character of $B$ trivial over $Z_0(H)$,
\end{itemize}
the inequality $\deg(\iota_*E_B(\chi))\geq 0$ (respectively,
$\deg(\iota_* E_B(\chi))>0$) holds, where $E_B(\chi)$ is the
holomorphic line bundle over $U$ associated to the principal
$B$-bundle $E_B$ for the character $\chi$. 
\end{definition}

It should be clarified that since the complex codimension of $U^c$ is
at least 2, 
the degree $\deg(\iota_*E_B(\chi))$ of the direct image
$\iota_*E_B(\chi)$ is well-defined.

\begin{definition}\label{ds2}
A $G$-equivariant holomorphic principal $H$-bundle $E_H$ on
$X\times G/P$ is called \textit{equivariantly semistable}
(respectively, \textit{equivariantly stable}) if the condition in
Definition~\ref{ds1} for semistability (respectively, stability) holds
for all quadruples $(B,U,E_B,\chi)$ with the extra condition that both 
$U\subset X\times G/P$ and $E_B\subset E_H\vert_U$ are $G$-invariant.
\end{definition}

\begin{lemma}\label{le-sst}
A $G$-equivariant holomorphic principal $H$-bundle $E_H$ on
$X\times G/P$ is semi\-stable if and only if it is equivariantly
semistable.
\end{lemma}

\begin{proof}
If $E_H$ is semistable, then clearly it is equivariantly semistable.
To prove the converse, assume that $E_H$ equivariantly semistable.
Let $E_B\,\subset\, E_H\vert_U$ be the Harder--Narasimhan reduction of
$E_H$ (see \cite[Theorem 1]{AAB}). From the uniqueness of the
Harder--Narasimhan reduction it follows that both $U$ and $E_B$ are
preserved by the actions of $G$ on $X\times G/P$ and $E_H$
respectively. Since $E_H$ is equivariantly semistable, from the
conditions on the Harder--Narasimhan reduction it follows immediately
that $B=H$. This implies that $E_H$ is semistable. (See also~\cite[Lemma 4.1]{Bi}.)
\end{proof}

Let $E_H$ be a holomorphic principal $H$-bundle over $X\times G/P$. A
holomorphic reduction $E_B\subset E_H$ of structure group to some
parabolic subgroup $B\subset H$ is called \emph{admissible} if
\[
\deg(E_B(\chi))=0
\]
for each character $\chi$ of $B$ which is trivial on $Z_0(H)$, where
$E_B(\chi)$ is the line bundle over $X\times G/P$ associated to the
principal $B$-bundle $E_B$ for the character $\chi$.

\begin{definition}\label{ds3}
A holomorphic principal $H$-bundle $E_H$ on $X\times G/P$
is called \emph{polystable} if either $E_H$ is stable,
or there is a proper parabolic subgroup $B\subset H$ and a
holomorphic reduction of structure group
\[
E_{L(B)}\subset E_H
\]
of $E_H$ to a Levi subgroup $L(B)\subset B$ over $X\times G/P$, such
that the following two conditions are satisfied:
\begin{itemize}
\item the principal  $L(B)$-bundle $E_{L(B)}$ is
stable, and
\item the reduction of structure group of $E_H$ to $B$, obtained by
extending the structure group of $E_{L(B)}$ using the inclusion of $L(B)$
in $B$, is admissible.
\end{itemize}
\end{definition}

\begin{definition}\label{ds4}
A $G$-equivariant holomorphic principal $H$-bundle $E_H$ on
$X\times G/P$ is called \emph{equivariantly polystable} if
either $E_H$ is equivariantly stable, or there is a proper
parabolic subgroup $B \,\subset\, H$ and a $G$-equivariant
holomorphic reduction of structure group
\[
E_{L(B)}\subset E_H
\]
of $E_H$ to a Levi subgroup $L(B)\subset B$ over $X\times G/P$ such
that the following two conditions are satisfied:
\begin{itemize}
\item the principal $L(B)$-bundle $E_{L(B)}$ is equivariantly stable,
  and
\item the reduction of structure group of $E_H$ to $B$, obtained by
  extending the structure group of $E_{L(B)}$ using the inclusion of
  $L(B)$ in $B$, is admissible.
\end{itemize}
\end{definition}

\begin{lemma}\label{le-pst}
A $G$-equivariant holomorphic principal $H$-bundle $E_H$ on
$X\times G/P$ is polystable if and only if it is equivariantly
polystable.
\end{lemma}

\begin{proof}
The proof is identical to the proof of Lemma 4.2 of \cite{Bi};
we omit the details.
\end{proof}

\subsection{Semistable and polystable torsors in holomorphic quiver bundles}
\label{sub:ss-P-equiv}

Our definition of semistability and polystability of $H$-torsors in
$(Q,\cK)$-bundles over $X$ will rely on the corresponding notions for
quiver bundles. To define the latter, we need to enlarge the category
of holomorphic $(Q,\cK)$-bundles, considering the so-called
$(Q,\cK)$-sheaves on $X$. To do this, we simply replace the notion of
holomorphic vector bundle by coherent sheaf of $\cO_X$-modules in the
corresponding definitions of~\secref{sub:Q-bundles} (see~\cite[\S
2.1]{AG2} for details). In particular, a \emph{$Q$-sheaf} on $X$ is a
pair $(\bE,\bphi)$ consisting of a set $\bE$ of coherent sheaves
$E_\lambda$ on $X$, indexed by $\lambda\in Q_0$, and a set $\bphi$ of
sheaf homomorphisms $\phi_a\colon E_{ta}\to E_{ha}$, indexed by $a\in
Q_1$. Then $Q$-sheaves form an abelian category. Note that if a
$Q$-sheaf satisfies the set of relations $\cK$, then so do all its
$Q$-subsheaves. A $Q$-sheaf $(\bE,\bphi)$ is called \emph{torsion
  free} if so is $E_\lambda$, for all $\lambda\in Q_0$.

Fix a K\"ahler form $\omega$ on $X$, and $\bepsilon$ as
in~\eqref{eq:bepsilon}.
For each $\lambda\in Q_0$, we define the numbers
\begin{align*}&
n_\lambda\defeq\dim_{\mathbb{C}} M_\lambda\in\mathbb{N},
\\&
\tau'_\lambda\defeq -n_\lambda\sum_{\alpha\in \Delta_+(\glr)}
\epsilon^{-1}_\alpha\,\kappa(\lambda,\alpha^\vee)
\in\mathbb{R},
\end{align*}
(cf.~\cite[p. 36, (4.10)]{AG2}), where $M_\lambda$ is any irreducible
$L$-module in the isomorphism class $\lambda$
(see~\secref{sub:quiver-relation}), $\kappa$ is the Killing form on
$\glg$, $\alpha^\vee\defeq 2\alpha/\kappa(\alpha,\alpha)$, and
$\epsilon_\alpha$ is defined by
\[
\epsilon_\alpha\defeq\sum_{\beta\in \Sigma}
\epsilon_\beta\,\kappa(\lambda_\beta,\alpha^\vee),
\]
for $\alpha\in\Delta_+(\glr)\setminus\Sigma$.
The \emph{$\tau'$-degree} and the \emph{$\tau'$-slope} of a
torsion-free $Q$-sheaf $(\bE,\bphi)$ on $X$ are 
\begin{gather*}
\deg_{\tau'}(\bE,\bphi)\defeq
\sum_{\lambda\in Q_0}(n_\lambda\deg E_\lambda -\tau'_\lambda\rk E_\lambda),
\\
\mu_{\tau'}(\bE,\bphi)\defeq\frac{\deg_{\tau'}(\bE,\bphi)}
{{\displaystyle\sum_{\lambda\in Q_0}n_\lambda\rk E_\lambda}},
\end{gather*}
respectively, where the degree is defined as in~\eqref{eq:deg}, and
$\rk E_\lambda$ is the rank of $E_\lambda$.

A torsion-free $Q$-sheaf $(\bE,\bphi)$ on $X$ is called
\emph{semistable} if all its non-zero $Q$-subsheaves $(\bE',\bphi')$
satisfy
\[
\mu_{\tau'}(\bE',\bphi')\leq\mu_{\tau'}(\bE,\bphi).
\]
It is called \emph{stable} if this inequality is strict for all proper
$Q$-subsheaves, and \emph{polystable} if it is a direct sum of stable
$Q$-sheaves of the same $\tau'$-slope (see \cite[Definition
4.19]{AG2}).


\begin{definition}\label{def1}
An $H$-torsor $\FF\colon\Rep{H}\to\Vect{Q,\cK}$ is called
\emph{polystable} (respectively, \emph{semistable}) if the holomorphic
$(Q,\cK)$-bundle $\FF(W)$ is polystable (respectively, semistable) for
every indecomposable $H$-module $W$.
\end{definition}

The above definitions of $\tau'$-degree and $\tau'$-slope are
motivated by the following.

\begin{lemma}\label{lem:slopes}
Let $E$ be a $P$-equivariant holomorphic vector bundle $E$ over $X$,
and 
\[
E'\defeq G\times^PE
\]
the induced $G$-equivariant holomorphic vector bundle over $X\times
G/P$ (see Theorem~\ref{prop:cat-Q-bundles}). Then the degree and the
slope of $E'$ with respect to the K\"ahler form $p^*_1\omega
+p^*_2\omega_\bepsilon$ are respectively computed by the
$\tau'$-degree and the $\tau'$-slope of the holomorphic
$(Q,\cK)$-bundle $(\bE,\bphi)$ over $X$ associated to $E$:
\[
\deg E'=\deg_{\tau'}(\bE,\bphi), \quad \mu(E')=\mu_{\tau'}(\bE,\bphi).
\]
\end{lemma}

\begin{proof}
Fix an irreducible $L$-module $M_\lambda$ in the isomorphism class
$\lambda$, for each $\lambda\in Q_0$, as
in~\secref{sub:quiver-relation}. Then $M_\lambda$ can also be
considered as an irreducible $P$-module, with the unipotent radical
$R_u(P)\subset P$ acting trivially (see~\cite[\S 1.1.3]{AG2}). Let
$\cO_\lambda=G\times^PM_\lambda$ be its associated homogeneous
holomorphic vector bundle over $G/P$ (see
Lemma~\ref{lem:ind-red-homgeneous-vect}). By~\cite[p. 36,
Lemma~4.15]{AG2} and the definition of $\tau'_\lambda$, the slope
of $\cO_\lambda$ with respect to $\omega_\bepsilon$ is
\[
\mu_\bepsilon(\cO_\lambda)\defeq\frac{\deg_\bepsilon\cO_\lambda}{\rk\cO_\lambda}=-\frac{\tau'_\lambda}{n_\lambda}
\]
where we use the normalization convention in~\eqref{eq:deg} for the
degree $\deg_\bepsilon\cO_\lambda$ of $\cO_\lambda$ with respect to
$\omega_\bepsilon$. Furthermore, applying the induction process of
Theorem~\ref{prop:cat-Q-bundles} (cf.~\cite[Lemma 2.8 and Theorem
2.5]{AG2}) to the image of~\eqref{eq:functor-S}, we see that $E'$ has
underlying vector bundle
\[
E'=\bigoplus_{\lambda\in Q_0}p^*_1E_\lambda\otimes p^*_2\cO_\lambda,
\]
so it has rank and degree (with respect to $p^*_1\omega
+p^*_2\omega_\bepsilon$) given by
\begin{align*}
\rk(E')&=\sum_{\lambda\in Q_0}\rk(p^*_1E_\lambda\otimes p^*_2\cO_\lambda)
=\sum_{\lambda\in Q_0}n_\lambda\rk{E_\lambda},\\
\deg E'&
=\sum_{\lambda\in Q_0}\deg(p^*_1E_\lambda\otimes p^*_2\cO_\lambda)
=\sum_{\lambda\in Q_0}n_\lambda(\deg{E_\lambda}+\mu_\bepsilon(\cO_\lambda)\rk{E_\lambda})
=\deg_{\tau'}(\bE,\bphi),
\end{align*}
respectively. The result about the slope now follows by diving these quantities.
\end{proof}

\subsection{Dimensional reduction, semistability and polystability}
\label{sub:preservation-ss}

\begin{theorem}\label{thm:preservation-ss}
Let $E'_H$ be a $G$-equivariant holomorphic principal $H$-bundle over
$X\times G/P$. Let $\FF\colon\Rep{H}\to\Vect{Q,\cK}$ be the
corresponding $H$-torsor in quiver bundles over $X$. Then $E'_H$ is
polystable (respectively, semistable) with respect to the K\"ahler
form $p^*_1\omega +p^*_2\omega_{\varepsilon}$ if and only if $\FF$ is
polystable (respectively, semistable).
\end{theorem}

\begin{proof}
Let $E_H$ be the $P$-equivariant principal $H$-bundle over $X$
corresponding to $\FF$ and $E'_H$ in Theorem~\ref{thm1} and
Corollary~\ref{cor1}. Let $W$ be an $H$-module. Then the functorial
correspondence
\[
E'_H\cong G\times^PE_H
\]
of equivariant principal bundles of
Lemma~\ref{lem:induction-reduction} (see~\eqref{eq:induced-pfb})
induces another one
\begin{equation}\label{eq:identifications-induced-vb}
E'_W\cong G\times^PE_W,
\end{equation}
where $E_W$ and $E'_W$ are the $P$-equivariant holomorphic vector
bundle on $X$ associated to the principal bundle $E_H$, and the
$G$-equivariant holomorphic vector bundle on $X\times G/P$ associated
to the principal bundle $E'_H$, respectively, for the $H$-module $W$,
that is
\begin{equation}\label{de-iii}
E_W\defeq E_H\times^HW, \quad E'_W\defeq E'_H\times^HW.
\end{equation}

We will apply the identification~\eqref{de-iii} to the indecomposable
$H$-submodules of the Lie algebra $\LieH$ of $H$. More precisely, the
adjoint representation of $H$ has a decomposition
\begin{equation}\label{eq:adj-rep}
\LieH=\bigoplus_{i=1}^\ell W_i
\end{equation}
into a direct sum of indecomposable $H$-modules. Since $H$ is
reductive, we know that each $W_i$ is self-dual, meaning the
$H$-module $W^*_i$ is isomorphic to $W_i$ for all $1\leq i\leq
\ell$. Then
\begin{equation}\label{de-ii}
\ad(E'_H)=\bigoplus_{i=1}^\ell E'_i, 
\end{equation}
where $\ad(E'_H)\defeq E'_H(\LieH)$ and $E'_i\defeq E'_H(W_i)$ are the
vector bundles over $X\times G/P$ associated to $E'_H$ for the
$H$-modules $\LieH$ and $W_i$, respectively. Note that the vector
bundle $(E'_i)^*$ is isomorphic to $E'_i$, because the $H$-module
$W_i$ is self-dual, so in particular,
\begin{equation}\label{de-i}
\deg(E'_i)=0.
\end{equation}

Assume now that $\FF$ is polystable. To prove that the principal
$H$-bundle $E'_H$ is polystable, note that if $W$ is an indecomposable
$H$-module, then $E'_W$ is polystable. This follows because the
holomorphic $(Q,\cK)$-bundle $\FF(W)$ is polystable in this case (by
Definition~\ref{def1}), so $E'_W\cong G\times^PE_W$ is equivariantly
polystable by \cite[p. 40, Theorem~4.21]{AG2}, and hence $E'_W$ is
polystable by Lemma~\ref{le-pst}. In particular, each $E'_i$ is
polystable, because each $W_i$ is indecomposable. Hence the adjoint
vector bundle $\ad(E'_H)$ is polystable (by~\eqref{de-ii}
and~\eqref{de-i}), and therefore so is the principal $H$-bundle
$E'_H$, by \cite[p. 224, Corollary 3.8]{AB}.

To prove the converse, assume $E'_H$ is polystable. Let $W$ be an
indecomposable $H$-module. As $W$ is indecomposable, $Z_0(H)$ acts on
$W$ as multiplication through a character of $Z_0(H)$, so the
associated vector bundle $E'_W$ is polystable (by~\cite[p. 224,
Theorem 3.9]{AB}), and hence the holomorphic $(Q,\cK)$-bundle $\FF(W)$
is polystable too, by~\cite[p. 40, Theorem 4.21]{AG2} and Lemma
\ref{le-pst}. Thus the $P$-equivariant principal $H$-bundle $E_H$ is
polystable.

The proof for semistability is very similar.

Assume $\FF$ is semistable. To prove that the principal $H$-bundle
$E'_H$ is semistable, we first observe that if $W$ is an
indecomposable $H$-module, then $E'_W$ is semistable. This follows
because the associated $P$-equivariant vector bundle $E_W$ corresponds
to a semistable $(Q,\cK)$-module $\FF(W)$ in this case, so it follows
as above, from \cite[p. 40, Theorem 4.21]{AG2} and Lemma \ref{le-sst},
that the associated vector bundle $E'_H(W)$ is semistable. In
particular, each vector bundle $E'_i$ in \eqref{de-i} is semistable,
because $W_i$ is indecomposable. Hence we conclude from~\eqref{de-ii}
and~\eqref{de-i} that the adjoint vector bundle $\ad(E'_H)$ is
semistable, and therefore so is the principal $H$-bundle $E'_H$ is
semistable, by \cite[p. 214, Proposition 2.10]{AB}.

To prove the converse, assume $E'_H$ is semistable. To prove that
$\FF$ is semistable, in view of \cite[p. 40, Theorem 4.21]{AG2} and
Lemma \ref{le-sst}, it suffices to show that for every indecomposable
$H$-module $W$, the associated vector bundle $E'_H(W)$ is
semistable. If the K\"ahler class on $X\times G/P$ is rational, then
$E'_H(W)$ is semistable by \cite[p. 285, Theorem 3.18]{RR}. For the
general K\"ahler class, the semistability of $E'_H(W)$ follows from
\cite[p. 700, Lemma 5]{AAB}. This completes the proof of the
proposition.
\end{proof}

\subsection{Dimensional reduction and the vortex equations}
\label{sub:preservation-HYM}

A \emph{hermitian metric} on a holomorphic $(Q,\cK)$-bundle
$(\bE,\bphi)$ over $X$ is a set $\bh=\{h_\lambda\}_{\lambda\in Q_0}$
of ($C^\infty$) hermitian metrics $h_\lambda$ on $E_\lambda$, indexed
by the vertices $\lambda\in Q_0$ (where we set $h_\lambda=0$ for all
$\lambda$ such that $E_\lambda=0$). For each arrow $a\in Q_1$, the
homomorphism $\phi_a\colon E_{ta}\to E_{ha}$ has a ($C^\infty$)
adjoint
\[
\phi_a^*\colon E_{ha}\longrightarrow E_{ta}
\]
with respect to the hermitian metrics $h_{ta}$ and $h_{ha}$ on
$E_{ta}$ and $E_{ha}$, respectively. Let $\Lambda$ be the adjoint of
multiplication of differential forms on $X$ by the K\"ahler form
$\omega$. Let $F_{h_\lambda}$ be the curvature of the Chern connection
of the hermitian metric $h_\lambda$ on $E_\lambda$, for all
$\lambda\in Q_0$ such that $E_\lambda\neq 0$.
A hermitian metric $\bh$ on $(\bE,\bphi)$ satisfies the \emph{quiver
  vortex equations} if
\[
n_\lambda\sqrt{-1}\,\Lambda F_{h_\lambda}
+\sum_{a\in h^{-1}(\lambda)} \phi_a\phi_a^*
-\sum_{a\in t^{-1}(\lambda)} \phi_a^*\phi_a
=\tau'_{\lambda}\Id_{E_\lambda},
\]
for all $\lambda\in Q_0$ such that $E_\lambda\neq 0$, where
$\Id_{E_\lambda}\colon E_\lambda\to E_\lambda$ is the identity map
(cf.~\cite[\S 2.1]{AG3}).

Let $\FF\colon\Rep{H}\to\Vect{Q,\cK}$ be an $H$-torsor in holomorphic
$(Q,\cK)$-bundles over $X$, and $E_H$ the corresponding
$P$-equivariant holomorphic principal $H$-bundle on $X$. We define the
\emph{adjoint holomorphic $(Q,\cK)$-bundle} of $\FF$ as the
holomorphic $(Q,\cK)$-bundle $\ad(\FF)$ on $X$ corresponding to the
($P$-equivariant) adjoint holomorphic vector bundle
\[
\ad(E_H)=E_H(\LieH)=E_H\times^H\LieH
\]
of $E_H$ via Theorem~\ref{prop:cat-Q-bundles} ($\LieH$ being the
adjoint representation of $H$). In other words, 
\begin{equation*}
\ad(\FF)\defeq\FF(\LieH)\in\Vect{Q,\cK}.
\end{equation*}

\begin{theorem}\label{thm2}
An $H$-torsor $\FF\colon\Rep{H}\to\Vect{Q,\cK}$ in holomorphic
$(Q,\cK)$-bundles over $X$ is polystable if and only if its
adjoint holomorphic $(Q,\cK)$-bundle $\ad(\FF)$ admits a hermitian
metric $\bh$ that satisfies the quiver vortex equations.
\end{theorem}

\begin{proof}
Let $E_H$ and $E'_H=G\times^PE_H$ be the $P$-equivariant principal
$H$-bundle on $X$ and the $G$-equivariant principal $H$-bundle on
$X\times G/P$ corresponding to $\FF$, respectively
(see~\eqref{eq:induced-pfb}). As in the proof of
Theorem~\ref{thm:preservation-ss}, using the adjoint
$H$-representation $\LieH$ and the indecomposable $H$-modules $W_i$
appearing in its decomposition~\eqref{eq:adj-rep}, the principal
bundles $E_H$ and $E'_H$ induce $P$-equivariant vector bundles over
$X$, and $G$-equivariant vector bundles over $X\times G/P$,
respectively given by
\begin{gather*}
\ad(E_H)=E_H\times^H\LieH=\bigoplus_{i=1}^\ell E_i, 
\quad
\ad(E'_H)=E'_H\times^H\LieH=\bigoplus_{i=1}^\ell E'_i,
\\
E_i=E_H\times^HW_i, 
\quad
E'_i=E'_H\times^HW_i,
\end{gather*}
that are related as in~\eqref{eq:identifications-induced-vb}, that is,
\[
\ad(E'_H)\cong G\times^P\ad(E_H),\quad E'_i\cong G\times^PE_i.
\]
Furthermore, the decomposition~\eqref{eq:adj-rep} induces another one
of the adjoint $(Q,\cK)$-bundle
\begin{equation}\label{eq:thm2.1}
\ad(\FF)=\FF(\LieH)=\bigoplus_{i=1}^\ell\,(\bE_i,\bphi_i),
\end{equation}
with $(\bE_i,\bphi_i)\defeq\FF(W_i)$ for all $1\leq i\leq\ell$, where
Lemma~\ref{lem:slopes} and~\eqref{de-i} imply
\begin{equation}\label{eq:thm2.2}
\mu_{\tau'}(\bE_i,\bphi_i)=\mu(E'_i)=0.
\end{equation}

Assume now that $\FF$ is polystable. Then $(\bE_i,\bphi_i)=\FF(W_i)$
is polystable, because $W_i$ is indecomposable, so $\ad(\FF)$ is
polystable, by~\eqref{eq:thm2.1} and~\eqref{eq:thm2.2}. Hence
$\ad(\FF)$ admits a hermitian metric $\bh$ that satisfies the quiver
vortex equations, by \cite[p. 41, Theorem 4.24]{AG2}.

To prove the converse, assume that $\ad(\FF)=\FF(\LieH)$ admits a
hermitian metric $\bh$ that satisfies the quiver vortex equations.
Then the $Q$-bundle $\ad(\FF)$ is polystable, by \cite[p. 41, Theorem
4.24]{AG2}, and hence the vector bundle $\ad(E'_H)$ is polystable,
by~\cite[p. 40, Theorem 4.21]{AG2}, because $\ad(\FF)$ corresponds to
the adjoint bundle $\ad(E'_H)$ of $E'_H$ via
Theorem~\ref{prop:cat-Q-bundles}. Hence $E'_H$ is polystable, by
\cite[p. 224, Corollary 3.8]{AB}, and therefore we conclude from
Theorem~\ref{thm:preservation-ss} that $E_H$ is polystable.
\end{proof}


\end{document}